\documentclass[reqno,12pt]{amsart}
\usepackage{amsthm,amsfonts,amssymb,euscript,
mathrsfs,graphics,color,amsmath,latexsym,marginnote}
\usepackage{dsfont}   
\usepackage{mathtools}

\theoremstyle{plain}

\usepackage{amsmath}
\usepackage{hyperref}
\usepackage{times}

\numberwithin{equation}{section}

\newcommand{\ta}{\mathtt{a}}

\newtheorem{theorem}{Theorem}[section]
\newtheorem{proposition}[theorem]{Proposition}
\newtheorem{lemma}[theorem]{Lemma}
\newtheorem{corollary}[theorem]{Corollary}

\newtheorem{remark}[theorem]{Remark}
\newtheorem{remarks}[theorem]{Remark}

\newtheorem{definition}[theorem]{Definition}
\newcommand{\be}{\begin{equation}}
\newcommand{\ee}{\end{equation}}

\newcommand{\eps}{\varepsilon}

\newcommand{\e}{\varepsilon}

\newcommand{\ov}{\overline}

\newcommand{\R}{\mathbb R}
\newcommand{\C}{\mathbb C}

\newcommand{\Z}{\mathbb Z}

\newcommand{\T}{\mathbb T}

\newcommand{\s }{\sigma }
\newcommand{\ii }{{\rm i} }

\newcommand{\m }{\mu }

\newcommand{\x }{\xi }

\newcommand{\pa}{\partial}

\def\bar{\overline}

\renewcommand{\Re}{\mathrm{Re}\,}
\renewcommand{\Im}{\mathrm{Im}\,}

\def\ba{\begin{aligned}}
\def\ea{\end{aligned}}
\def\beginm{\begin{multline}}
\def\endm{\end{multline}}

\providecommand{\vect}[2]{{\bigl[\begin{smallmatrix}#1\\#2\end{smallmatrix}\bigr]}}

\makeatletter

\def\l@subsection{\@tocline{2}{0pt}{2.5pc}{5pc}{}}
\def\l@subsubsection{\@tocline{3}{0pt}{4.5pc}{5pc}{}}
\renewcommand\tocchapter[3]{%
  \indentlabel{\@ifnotempty{#2}{\ignorespaces#2.\quad}}#3%
}
\newcommand\@dotsep{4.5}
\def\@tocline#1#2#3#4#5#6#7{\relax
  \ifnum #1>\c@tocdepth 
  \else
    \par \addpenalty\@secpenalty\addvspace{#2}%
    \begingroup \hyphenpenalty\@M
    \@ifempty{#4}{%
      \@tempdima\csname r@tocindent\number#1\endcsname\relax
    }{%
      \@tempdima#4\relax
    }%
    \parindent\z@ \leftskip#3\relax \advance\leftskip\@tempdima\relax
    \rightskip\@pnumwidth plus1em \parfillskip-\@pnumwidth
    #5\leavevmode\hskip-\@tempdima{#6}\nobreak
    \leaders\hbox{$\m@th\mkern \@dotsep mu\hbox{.}\mkern \@dotsep mu$}\hfill
    \nobreak
    \hbox to\@pnumwidth{\@tocpagenum{#7}}\par
    \nobreak
    \endgroup
  \fi}
\def\l@subsection{\@tocline{2}{0pt}{2.5pc}{5pc}{}}

\begin{document}

\title{Long-Time stability of the  quantum hydrodynamic system on irrational tori}

\author{Roberto Feola}
\address{\scriptsize{Dipartimento di Matematica, Universit\`a degli Studi di Milano, Via Saldini 50, I-20133}}
\email{roberto.feola@unimi.it}

\author{Felice Iandoli}
\address{\scriptsize{Laboratoire Jacques-Louis Lions, Sorbonne Universit\'e, UMR CNRS 7598\\
4, Place Jussieu\\
 75005 Paris Cedex 05, France}}
\email{felice.iandoli@sorbonne-universite.fr}

\author{Federico Murgante}
\address{\scriptsize{International School for Advanced Studies (SISSA), Via Bonomea 265, 34136, Trieste, Italy}}
\email{fmurgant@sissa.it}




\thanks{Felice Iandoli has been supported by ERC grant ANADEL 757996.}

\begin{abstract}
We consider the quantum hydrodynamic system
on a  $d$-dimensional  irrational torus with $d=2,3$.
We discuss the  behaviour, over a ``non-trivial’’ time interval,   
of the $H^s$-Sobolev norms of solutions.
More precisely we prove that, for generic irrational tori, the solutions, evolving form $\e$-small initial conditions, 
remain bounded in $H^s$ for a time scale of order $O(\e^{-1-1/(d-1)+})$, 
which is strictly larger with respect to
the time-scale provided by local theory.
We exploit a Madelung transformation to rewrite the system as a nonlinear  Schr\"odinger equation.
We therefore implement a Birkhoff normal form procedure involving small divisors
arising form \emph{three waves interactions}.
The main difficulty is to control the loss 
of derivatives coming from the exchange of energy between high Fourier modes. 
This is  due to the irrationality of the torus
which prevent to have ``good separation'' properties of the eigenvalues of the 
linearized operator at zero. The main steps of the proof are:
(i) to prove precise  lower bounds on small divisors; (ii) to construct a modified energy
by means of a suitable \emph{high/low} frequencies analysis, which gives an \emph{a priori}
estimate on the solutions.
\end{abstract}


\maketitle

\tableofcontents

\section{Introduction}
We consider the quantum hydrodynamic system on an irrational torus of dimension 2 or 3
\begin{equation}\tag{QHD}\label{EK3}
\left\{\begin{aligned}
& \pa_{t}\rho=-\mathtt{m}\Delta\phi-{\rm div}(\rho\nabla\phi)\\
&\pa_{t}\phi=-\tfrac{1}{2}|\nabla\phi|^{2}-g(\mathtt{m}+\rho)
+\frac{\kappa}{\mathtt{m}+\rho}\Delta\rho
-\frac{\kappa}{2(\mathtt{m}+\rho)^{2}}
|\nabla\rho|^{2}\,,
\end{aligned}\right.
\end{equation}
where $\mathtt{m}>0$, $\kappa>0$,   the function $g$ belongs to $C^{\infty}(\mathbb{R}_{+};\mathbb{R})$ and $g(\mathtt{m})=0$.
The function $\rho(t,x)$ is such that $\rho(t,x)+\mathtt{m}>0$ 
and it has zero average in $x$. The space variable $x$ belongs to the irrational torus
\begin{equation}\label{toriIrr}
\mathbb{T}^{d}_{{\nu}}:=(\mathbb{R}/ 2\pi \nu_1\mathbb{Z})\times 
\cdots\times(\mathbb{R}/ 2\pi \nu_d \mathbb{Z})\,,\qquad d=2,3\,,
\end{equation}
with ${\nu}=(\nu_1,\ldots,\nu_{d})\in [1,2]^{d}$. We assume the 
\emph{strong} 
ellipticity condition
\begin{equation}\label{elliptic}
g'(\mathtt{m})>0. 
\end{equation}

\noindent We shall consider an initial condition $(\rho_0,\phi_0)$ having small size 
$\varepsilon\ll1$ in the standard Sobolev space $H^s(\mathbb{T}^{d}_{\nu})$ with $s\gg1$. 
Since the equation has a \emph{quadratic nonlinear term}, the local existence theory (which may be obtained 
in the spirit of \cite{BMM1, Feola-Iandoli-local-tori}) 
implies that the solution of \eqref{EK3} remains of size $\varepsilon$ 
for times of magnitude $O(\varepsilon^{-1})$. 
The aim of this paper is to prove that, for \emph{generic irrational tori}, 
the solution remains of size $\varepsilon$ for longer times.

For $\phi\in H^{s}(\mathbb{T}^{d}_{\nu})$ we define
\begin{equation}\label{pizeroBot}
\Pi_{0}\phi:=\frac{1}{(2\pi)^{d}\nu_1\cdots\nu_{d}}\int_{\mathbb{T}^{d}_{\nu}}\phi(x)dx\,,\qquad
\Pi_0^{\perp}:={\rm id}-\Pi_0\,.
\end{equation}
Our main result is the following. 
\begin{theorem}\label{thm-main} Let $d=2$ or $d=3$.
There exists $s_0\equiv s_0(d)\in\R$ such that for almost all 
$\nu\in[1,{2}]^{d}$, for any $s\geq s_0$, $\mathtt{m}>0$, $\kappa>0$ there exist $C>0$,
 $\eps_0>0$ such that for any $0<\eps\leq \eps_0$
we have the following. 
For any initial data 
$(\rho_0,\phi_0)\in H_0^{s}(\mathbb{T}_{\nu}^{d})\times H^{s}(\mathbb{T}_{\nu}^{d})$
such that
\begin{equation}\label{initialstima}
\|\rho_0\|_{H^{s}(\mathbb{T}_{\nu}^{d})}+
\|\Pi_0^{\bot}\phi_0\|_{H^{s}(\mathbb{T}_{\nu}^{d})}\leq \eps\,,
\end{equation}
there exists a unique solution 
of \eqref{EK3} with $(\rho(0),\phi(0))=(\rho_0,\phi_0)$
 such that 
\begin{equation}\label{tesiKG}
\begin{aligned}
&(\rho(t),\phi(t))\in C^0\big([0,T_\eps);H^{s}(\T_\nu^d)\times H^{s}(\T_\nu^d)\big)
\bigcap 
C^1\big([0,T_\eps);H^{s-2}(\T_\nu^d)\times H^{s-2}(\T_\nu^d)\big)\,, \\
& \sup_{t\in[0,T_\eps)}\Big(\|\rho(t,\cdot)\|_{H^{s}(\mathbb{T}_{\nu}^{d})} +
\| \Pi_0^{\bot}\phi(t,\cdot)\|_{H^{s}(\mathbb{T}_{\nu}^{d})}    \Big)\leq C\eps\,, 
\qquad T_\eps\geq  \eps^{-1-\tfrac{1}{d-1}}\log^{-\frac{d+1}{2}}\big(1+\e^{\frac{1}{1-d}}\big)\,.
\end{aligned}
\end{equation}
\end{theorem}

\smallskip
\noindent
\textbf{Derivation from Euler-Korteweg system.} 
The  \eqref{EK3} is derived from the 
 compressible Euler-Korteweg system 
\begin{equation*}\tag{EK}\label{EK1}
\left\{
\begin{aligned}
&\pa_{t}\rho+{\rm div}(\rho\vec{u})=0\\
&\pa_{t}\vec{u}+\vec{u}\cdot\nabla\vec{u}+\nabla g(\rho)=
\nabla\big(K(\rho)\Delta \rho
+\tfrac{1}{2}K'(\rho)|\nabla\rho|^{2}\big)\,,
\end{aligned}\right.
\end{equation*}
where the function $\rho(t,x)>0$ is the density of the fluid
and $\vec{u}(t,x)\in \mathbb{R}^{d}$ is the time dependent velocity field;
we assume that $K(\rho), g(\rho)\in C^{\infty}(\mathbb{R}_{+};\mathbb{R})$
and that $K(\rho)>0$.
In particular, in \eqref{EK3}, we assumed
\begin{equation}\label{hypKK}
K(\rho)=\frac{\kappa}{\rho}\,,\qquad \kappa\in \mathbb{R}_{+}\,.
\end{equation}
We look for solutions $\vec{u}$ which stay irrotational for all times, i.e.
\begin{equation}\label{hypUU}
\vec{u}=\vec{c}(t)+\nabla\phi\,,\quad  \vec{c}(t)\in \mathbb{R}^{d}\,,
\quad 
\vec{c}(t)=\frac{1}{(2\pi)^{d}\nu_1\cdots\nu_{d}}\int_{\mathbb{T}^{d}_{\nu}}\vec{u}dx\,,
\end{equation}
where $\phi:\mathbb{T}^{d}_{\nu}\to \mathbb{R}$ is a scalar potential.
By the second equation in \eqref{EK1} 
and using that ${\rm rot} \vec{u}=0$ we deduce
\[
\pa_{t}\vec{c}(t)=-\frac{1}{(2\pi)^{d}\nu_1\cdots\nu_d}
\int_{\mathbb{T}^{d}_{\nu}}\vec{u}\cdot\nabla\vec{u}dx=0
\qquad \Rightarrow\qquad \vec{c}(t)=\vec{c}(0)\,.
\]
The system \eqref{EK1} is Galilean invariant, i.e. if $(\rho(t,x), \vec{u}(t,x))$ 
solves \eqref{EK1} then also
\[
\rho_{\vec{c}}(t,x):=\rho(t,x+\vec{c}t)\,,
\qquad 
\vec{u}_{\vec{c}}(t,x):=\vec{u}(t,x+\vec{c}t)-\vec{c}\,,
\]
solves \eqref{EK1}. 
Then we can always assume that $\vec{u}=\nabla \phi$
for some scalar potential $\phi:\mathbb{T}^{d}_{\nu}\to \mathbb{R}$. The system \eqref{EK1}
reads
\begin{equation}\label{EK2}
\left\{\begin{aligned}
& \pa_{t}\rho+{\rm div}(\rho\nabla\phi)=0\\
&\pa_{t}\phi+\tfrac{1}{2}|\nabla\phi|^{2}+g(\rho)=K(\rho)\Delta\rho
+\tfrac{1}{2}K'(\rho)|\nabla\rho|^{2}\,.
\end{aligned}\right.
\end{equation}
Notice that the average
\begin{equation}\label{averRho}
\frac{1}{(2\pi)^{d}\nu_1\cdots\nu_d}\int_{\mathbb{T}^{d}_{\nu}}\rho(x)dx=\mathtt{m}\in \mathbb{R}\,,
\end{equation}
is a constant of motion of \eqref{EK2}. 
Notice also that the vector field  of \eqref{EK2}
depends only on $\Pi_0^{\perp}\phi$ (see \eqref{pizeroBot}).
In view of \eqref{averRho} we rewrite 
$\rho\rightsquigarrow\mathtt{m}+\rho$ where $\rho$ is a function
with zero average. 
Then,
the system \eqref{EK2}  (recall also \eqref{hypKK})  becomes \eqref{EK3}.\\

\smallskip
\noindent
{\bf Phase space and notation.}
In the paper we work with  functions 
belonging to the Sobolev space
\begin{equation}\label{spazioSob}
H^{s}(\mathbb{T}_{\nu}^d):=
\Big\{
u(x) = \frac{1}{(2\pi)^{d/2}}
\sum_{j \in \Z_{\nu}^{d} } u_{j} e^{\ii j\cdot x }\,\;:\;\;
\|u(\cdot)\|_{H^{s}(\mathbb{T}_{\nu}^{d})}^{2}:=\sum_{j\in \mathbb{Z}_{\nu}^{d}}\langle j\rangle^{2s}|u_{j}|^{2}<+\infty
\Big\}\,,
\end{equation}
where $\langle j \rangle:=\sqrt{1+|j|^{2}}$ for $j\in \mathbb{Z}_{\nu}^{d}$ with 
$\mathbb{Z}_{\nu}^{d}:=(\mathbb{Z}/\nu_1)\times\cdots\times(\mathbb{Z}/{\nu_d})$.
The natural phase space 
for \eqref{EK3} is 
$H_0^{s}(\mathbb{T}^{d}_{\nu})\times\dot{H}^{s}(\mathbb{T}^{d}_{\nu})$
where ${\dot H}^s (\T^{d}_{\nu}):= H^s (\T^{d}_{\nu}) \slash_{ \sim}$ 
is the homogeneous Sobolev space
obtained by the equivalence relation
$\psi_1 (x) \sim \psi_2 (x)$ if and only if $ \psi_1 (x) - \psi_2 (x) = c $ is a constant;   $H_0^{s}(\mathbb{T}^{d}_{\nu})$ is
the subspace of $H^{s}(\mathbb{T}^{d}_{\nu})$ of functions with zero average.
Despite this fact we prefer to work with a couple of variable 
$(\rho,\phi)\in H_0^s(\T_\nu^d)\times H^s(\T^d_\nu)$ 
but at the end we control only the norm 
$\| \Pi_0^{\bot} \phi \|_{H^{s}(\mathbb{T}_{\nu}^{d})}$ 
which in fact is the relevant quantity for \eqref{EK3}.  
To lighten the notation we shall write $\|\cdot\|_{H^{s}_{\nu}}$ to denote $\|\cdot\|_{H^{s}(\mathbb{T}_{\nu}^{d})}$.

In the following we will use the notation $A\lesssim B$ to denote 
$A\le C B$ where $C$ is a positive constant
depending on  parameters fixed once for all, for instance $d$
 and $s$.
 We will emphasize by writing $\lesssim_{q}$
 when the constant $C$ depends on some other parameter $q$.

\bigskip
\textbf{Ideas of the proof.}
The general  \eqref{EK1} is a system of \emph{quasi-linear} equations. The  case \eqref{EK3}, i.e. the system \eqref{EK1} with the particular choice \eqref{hypKK}, reduces, for \emph{small solutions}, to a \emph{semi-linear} equation, more precisely to a nonlinear Schr\"odinger equation. This is a consequence of the fact that the \emph{Madelung transform} (introduced for the first time in the seminal work by Madelung \cite{Madelung}) is well defined for small solutions. In other words  one  can introduce the new variable  $\psi:= \sqrt{\mathtt{m}+\rho} e^{{{\rm i} \phi}/{\hbar}}$ (see Section \ref{madmad} for details), where $\hbar=2\sqrt{k}$, obtaining the equation
\begin{equation*}
 \pa_{t}\psi=\ii \Big(\frac{\hbar}{2}\Delta\psi-\frac{1}{\hbar}g(|\psi|^{2})\psi\Big)\,.
\end{equation*}
Since $g(\mathtt{m})=0$, such an equation has an   equilibrium point at $\psi=\sqrt{\mathtt{m}}$. The study of the stability of small solutions for \eqref{EK3} is equivalent to the study of the stability of the variable $z=\psi-\sqrt{\mathtt{m}}$. The equation for the variable $z$ reads
\begin{equation*}
\pa_{t}z=-\ii \big(\frac{\hbar |D|_{\nu}^{2}}{2}+\frac{\mathtt{m}g'(\mathtt{m})}{\hbar}\big)z
-\ii \frac{\mathtt{m}g'(\mathtt{m})}{\hbar}\bar{z}+
f(z),
\end{equation*}
where $f$ is a smooth function having a zero of order 2 at $z=0$, i.e. $|f(z)|\lesssim |z|^2$, and $|D|_{\nu}^{2}$ is the Fourier multiplier with symbol
\begin{equation}\label{gatto25}
|\x|_{\nu}^{2}:=\sum_{i=1}^{d}a_i|\x_{j}|^{2}\,,\quad a_i:=\nu_i^{2}\,,\quad \forall\,\x\in \mathbb{Z}^{d}\,.
\end{equation}
The aim is to use a \emph{Birkhoff normal form/modified energy} technique in order to reduce the size of the nonlinearity $f(z)$. To do that, it is convenient to perform  some preliminary reductions. First of all we want to eliminate the addendum $-\ii \frac{\mathtt{m}g'(\mathtt{m})}{\hbar}\bar{z}$. In other words we want to
diagonalize the matrix 
\begin{equation}\label{linearizz}
\mathcal{L}=\begin{pmatrix} 
\frac{\hbar}{2} |D|_{\nu}^2+ \frac{1}{\hbar}\mathtt{m}g'(\mathtt{m})
& 
\frac{1}{\hbar}\mathtt{m}g'(\mathtt{m})\vspace{0.3em}\\ 
\frac{1}{\hbar}\mathtt{m}g'(\mathtt{m}) & 
 \frac{\hbar}{2} |D|_{\nu}^2+\frac{1}{\hbar}\mathtt{m}g'(\mathtt{m})
\end{pmatrix} \,.
\end{equation}
To achieve the diagonalization of this matrix it is necessary to rewrite the equation in a system of coordinates which does not involve the zero mode. We perform this reduction in Section \ref{eliminazero}: we use the \emph{gauge invariance} of the equation as well as the $L^2$ norm preservation to eliminate the dynamics of the zero mode. This  idea has been introduced for the first time in \cite{Faouplane}. After the diagonalization of the matrix in \eqref{linearizz} we end up with a diagonal, quadratic, semi-linear equation with dispersion law  
\begin{equation*}
\omega(j):=\sqrt{ \frac{\hbar^2}{4} |j|_{\nu}^4+ \mathtt{m}g'(\mathtt{m}) |j|_{\nu}^2}\,,
\end{equation*}
where $j$ is a vector in $\mathbb{Z}^d\setminus \{0\}$. At this point we are ready to define a suitable modified energy. Our primary aim is to control the derivative of the $H^s$-norm of the solution 
\begin{equation}\label{derivata}
\frac{d}{dt}\|\tilde{z}(t)\|_{H^{s}}^2\,,
\end{equation}
where $\tilde{z}$ is the variable of the diagonalized system, for the longest time possible. 
Using the equation, such a quantity may be rewritten as the sum of trilinear expressions in $\tilde{z}$.  
We  perturb the Sobolev energy by expressions homogeneous of degree at least $3$ 
such that their time derivatives cancel out the main contribution 
(i.e. the one coming from cubic terms) in \eqref{derivata}, up to remainders of higher order. 
In trying to do this small dividers appear, i.e. denominators of the form
\begin{equation*}
\pm\omega(j_1)\pm\omega(j_2)
\pm\omega(j_3)\,.
\end{equation*}
It is fundamental that the perturbations we define is bounded by some power of $\|\tilde{z}\|_{H^{s}}$, 
with the same $s$ in \eqref{derivata}, otherwise we obtain an estimate with 
\emph{loss of derivatives}. Therefore we need to impose some lower bounds on the small dividers. 
Here it enters in the game the irrationality of the torus $\nu$. 
We prove indeed that for almost any $\nu\in[1,2]^{d}$, there exists  $\gamma>0$ such that
\begin{equation*}
|\pm\omega(j_1)\pm\omega(j_2)
\pm\omega(j_3)|\geq
\frac{\gamma}{\mu_1^{d-1}\log^{\,d+1}{\big(1+\mu_1^2\big)}\mu_3^{M(d)}}\,,
\end{equation*}
if $\pm j_1\pm j_2\pm j_3=0$,
we denoted by $M(d)$  a positive constant depending on the dimension $d$ and $\mu_i$ 
the $i$-st largest integer among $|j_1|, |j_2|$ and $|j_3|$. 
It is nowadays well known, see for instance \cite{Bambu,BG}, 
that the power of $\mu_3$ is not dangerous if we work in $H^s$ with $s$ big enough. 
Unfortunately we have also a power of the highest frequency 
$\mu_1$ which represents, in principle, a \emph{loss of derivatives}. 
However, this loss of derivatives may be transformed in a loss of length of 
the lifespan through partition of frequencies, 
as done for instance in \cite{Delort-Tori,IPtori,FGI20,BFGI}. 

 \bigskip
\textbf{Some comments.} As already mentioned, an estimate on small divisors involving 
only powers of $\mu_3$ is not dangerous. 
We may obtain such an estimate when  the equation is considered 
on the squared torus $\mathbb{T}^d$, using as a parameter 
the \emph{mass} $\mathtt{m}$. 
In this case, indeed, one can obtain better estimates by following the proof in \cite{Faouplane}. 
This is a consequence of the fact that the set of differences of eigenvalues is \emph{discrete}.
This is not the case of irrational tori with fixed mass, where
the set of eigenvalues is \emph{not discrete}.
Having estimates involving only $\mu_3$ one could actually prove an $\emph{almost-global}$ stability. More precisely one can prove, for instance, that there exists a zero Lebesgue measure set $\mathcal{N}\subset [1,+\infty)$, such that if $\mathtt{m}$ is in $[1,+\infty)\setminus \mathcal{N}$, then for any $N\geq 1$ if the initial condition is sufficiently regular (w.r.t. $N$) and of size $\varepsilon$ sufficiently small (w.r.t. $N$) then the solution stays of size $\varepsilon$ for a time of order $\varepsilon^{-N}$. The proof follows the lines of classical papers such as \cite{Bambu, BG, BDGS} by using the Hamiltonian structure of the equation. More precisely,   the system 
\eqref{EK3} can be written in the form
\begin{equation}\label{EK2Ham}
\pa_{t}\vect{\rho}{\phi}=X_{H}(\rho,\phi)
=\left(
\begin{matrix}
\pa_{\phi}H(\rho,\phi)\\
-\pa_{\rho}H(\rho,\phi)
\end{matrix}
\right)\,,
\end{equation}
where $\pa$ denotes the $L^2$-gradient and 
$H(\rho,\phi)$ is the Hamiltonian function 
\begin{equation}\label{HamFunc}
H(\rho,\phi)=\frac{1}{2}\int_{\mathbb{T}_{\nu}^d}(\mathtt{m}+\rho)|\nabla\phi|^{2}dx+
\int_{\mathbb{T}_{\nu}^d}\Big(\frac{1}{2}\frac{\kappa}{\mathtt{m}+\rho}|\nabla\rho|^{2}+G(\mathtt{m}+\rho)\Big)dx
\end{equation}
where $G'(\rho)=g(\rho)$. \\
We  do not know if the solution of \eqref{EK3} are globally defined. There are positive answers in the case that the equation is posed on the Euclidean space $\mathbb{R}^d$ with $d\geq 3$, see for instance \cite{Audiard}. Here the dispersive character of the equation is taken into account. For recent developments in this direction see \cite{Antonelli} and reference therein. It is worth mentioning also the scattering result for the Gross-Pitaevsii equation \cite{naka}.
Since we are considering the equation on a compact manifold, the dispersive estimates are not available. \\
It would be interesting to obtain a long time stability result also for solutions of the general system \eqref{EK1}. In this case the equation may be not recasted as a \emph{semi-linear} Schrödinger equation. Being a quasi-linear system, we expect that a para-differential approach, in the spirit of \cite{IPtori,FGI20} should be applied. However, in this case, the quasi-linear term is quadratic, hence \emph{big}. In \cite{IPtori,FGI20} the quasi-linear term is smaller. Therefore new ideas have to be introduced in order to improve the local existence theorem.

By using para-compositions (in the spirit of \cite{BD,Feola-Iandoli-Long,Feola-Iandoli-Totale}), in the case $d=1$, i.e. on the torus $\mathbb{T}^1$, it is possible to obtain  stronger results. This is the argument of a future work of one of us  with other collaborators \cite{murgantefuturo}.

\bigskip



\section{From (\ref{EK3}) to Nonlinear Schr\"odinger}\label{madmad}

\subsection{Madelung transform}
For $\lambda\in \mathbb{R}_{+}$, 
we define the change of variable (\emph{Madelung transform})
\begin{equation*}\label{mad}\tag{$\mathcal{M}$}
 \psi:= \mathcal{M}_{\psi}(\rho,\phi):=\sqrt{\mathtt{m}+\rho} e^{{\rm i}\lambda \phi}\,,
 \quad \bar{\psi}:=\mathcal{M}_{ \bar\psi}(\rho,\phi):=\sqrt{\mathtt{m}+\rho}e^{-\ii\lambda\phi}\,.
  \end{equation*}
  Notice that the inverse map has the form
    \begin{equation}\label{madelunginv}
  \mathtt{m}+\rho=\mathcal{M}_\rho^{-1}(\psi,\bar\psi):=|\psi|^{2}\,,\qquad \phi=\mathcal{M}_\phi^{-1}(\psi,\bar\psi):=\frac{1}{\lambda}
  \arctan\left(\frac{-\ii(\psi-\bar{\psi})}{\psi+\bar{\psi}}\right)\,.
  \end{equation}
  In the following lemma we provide a well-posedness result for the Madelung transform.
  \begin{lemma}\label{federico1}
  Define
  \begin{equation}\label{alberobello}
  \frac{1}{4\lambda^{2}}=\kappa\,,\qquad \hbar:= \frac{1}{\lambda}=2\sqrt{\kappa}\,.
  \end{equation}
  The following holds.
  
  \noindent
  $(i)$ Let $s>1$ and 
  \[
  \delta:=\frac{1}{\mathtt{m}}\|\rho\|_{H_{\nu}^{s}}+\frac{1}{\sqrt{\kappa}}\|\Pi_0^{\bot}\phi\|_{H_{\nu}^{s}} 
  \quad \sigma:=\Pi_0\phi\,.
  \]
  There is $C=C(s)>1$ such that, if $C(s)\delta\leq1$, then the function 
  $\psi$ in \eqref{mad} satisfies
  \begin{equation}\label{stimamand}
  \|\psi-\sqrt{\mathtt{m}}e^{\ii\sigma}\|_{H_{\nu}^{s}}\leq 2\sqrt{\mathtt{m}}\delta\,.
  \end{equation}
  
  \noindent
  $(ii)$ Define
  \[
  \delta':=\inf_{\sigma \in \T}\|\psi-\sqrt{\mathtt{m}}e^{\ii \sigma}\|_{H_{\nu}^{s}}\,.
  \]
  There is $C'=C'(s)>1$ such that, if $C'(s) \delta'(\sqrt{\mathtt{m}})^{-1}\leq 1$, then the functions $\rho,$
\begin{equation}\label{stimamand2}
\frac{1}{\mathtt{m}}\|\rho\|_{H_{\nu}^{s}}
+\frac{1}{\sqrt{\kappa}}\|\Pi_0^{\bot}\phi\|_{H_{\nu}^{s}}\leq 8\frac{1}{\sqrt{\mathtt{m}}}\delta'\,.
\end{equation}

  \end{lemma}
  \begin{proof}
  The bound \eqref{stimamand} follows by \eqref{mad} and classical estimates on composition operators 
  on Sobolev spaces (see for instance \cite{Mos66}).
  Let us check the \eqref{stimamand2}. 
  By the first of \eqref{madelunginv}, for any $\sigma \in \T$,  we have
  \begin{align}
  \| \rho\|_{H_{\nu}^s}  \leq&\|\sqrt{\mathtt{m}}(\psi e^{-\ii\sigma}-\sqrt{\mathtt{m}})\|_{H_{\nu}^s}
  +\| \sqrt{\mathtt{m}}(\ov{\psi}e^{\ii\sigma}-\sqrt{\mathtt{m}})\|_{H_{\nu}^s}
  +\|(\psi e^{-\ii\sigma}-\sqrt{\mathtt{m}})(\ov{\psi}e^{\ii\sigma}-\sqrt{\mathtt{m}})\|_{H_{\nu}^s}
  \\
  \leq& 
  \mathtt{m} \big( \frac{2}{\sqrt{\mathtt{m}}}\| \psi -\sqrt{\mathtt{m}}e^{\ii\sigma}\|_{H_{\nu}^s}
  +(\frac{1}{\sqrt{\mathtt{m}}}\| \psi -\sqrt{\mathtt{m}}e^{\ii\sigma}\|_{H_{\nu}^s})^2\big)\,.
  \end{align}
  Therefore, by the arbitrariness of $\sigma$ and using that $(\sqrt{\mathtt{m}})^{-1}\delta'\ll1$, one deduces
  \[
  \frac{1}{\mathtt{m}}\|\rho\|_{H_{\nu}^{s}}\leq 3\frac{1}{\sqrt{\mathtt{m}}}\delta'\,.
  \]
Moreover we note that
  \[
  \|(\psi-\bar{\psi})(\psi+\ov{\psi})^{-1}\|_{H_{\nu}^{s}}\leq 
  2\frac{1}{\sqrt{\mathtt{m}}}\| \psi- \sqrt{\mathtt{m}}\|_{H_{\nu}^{s}}\,.
  \]
  Then by the second in \eqref{madelunginv}, \eqref{alberobello},
  composition estimates on Sobolev spaces and the smallness
  condition $(\sqrt{\mathtt{m}})^{-1}\delta'\ll1$, one deduces, for any $\sigma \in \T$ 
  such that  $(\sqrt{\mathtt{m}})^{-1}\| \psi-\sqrt{\mathtt{m}}e^{\ii\sigma}\|_{H_{\nu}^s}\ll1$, that
  \begin{align*}
  \frac{1}{\sqrt{\kappa}}\|\Pi_0^{\bot}\phi\|_{H_{\nu}^{s}}&+
  2\| \Pi_0^{\bot}\arctan{ \big(\frac{-\ii (\psi- \ov{\psi})}{\psi+\ov{\psi}}\big)}\|_{H_{\nu}^s}
  \\
  &=2\| \Pi_0^{\bot}\arctan{ \big(\frac{-\ii (\psi e^{-\ii\sigma}- \ov{\psi}e^{\ii\sigma})}{\psi e^{-\ii\sigma}
  +\ov{\psi}e^{-\ii \sigma}}}\big)\|_{H_{\nu}^s}
  \\
  &\leq \frac{5}{\sqrt{\mathtt{m}}} \| \psi-\sqrt{\mathtt{m}} e^{\ii\sigma}\|_{H_{\nu}^s}\,.
  \end{align*}
  Therefore the \eqref{stimamand2} follows.
  \end{proof}
  We now rewrite equation \eqref{EK3} in the variable $(\psi,\bar{\psi})$.
  \begin{lemma}\label{gatto23}
  Let $(\rho,\phi)\in H^s_0(\T_\nu^d)\times H^s(\T_\nu^d)$ be a solution of \eqref{EK3} defined over a time interval $[0,T]$, $T>0$,
  such that
  \begin{equation}\label{hawaii20}
  \sup_{t\in[0,T)}\Big(\frac{1}{\mathtt{m}}\|\rho(t,\cdot)\|_{H_{\nu}^{s}} +\frac{1}{\sqrt{\kappa}}
\| \Pi_0^{\bot}\phi(t,\cdot)\|_{H_{\nu}^{s}}    \Big)\leq \eps
  \end{equation}
  for some $\e>0$ small enough. Then 
  the function $\psi$ defined in \eqref{mad} solves
  \begin{equation}\label{papsipsi4}
    \begin{cases}
 \pa_{t}\psi=-\ii \big(-\frac{\hbar}{2}\Delta\psi+\frac{1}{\hbar}g(|\psi|^{2})\psi\big)\\
 \psi(0)= \sqrt{\mathtt{m}+\rho(0)}e^{\ii \phi(0)}\,.
 \end{cases}
 \end{equation}
  \end{lemma}
    \begin{remark}
   We remark that the assumption of Lemma \ref{gatto23} 
   can be verified  in the same spirit of the local
well-posedness results in \cite{Feola-Iandoli-local-tori} and \cite{BMM1}. 
 \end{remark}
 
  \begin{proof}[Proof of Lemma \ref{gatto23}]
  The smallness condition \eqref{hawaii20} implies that the function $\psi$ is well-defined and satisfies 
  a bound like \eqref{stimamand}.
We first note
 \begin{align}
 &\nabla \psi = 
 ( \frac{1}{2\sqrt{\mathtt{m}+\rho}}\nabla \rho
 + {\rm i}\lambda  \sqrt{\mathtt{m}+\rho} \nabla \phi ) e^{{\rm i}\lambda\phi}\,, \label{winston}
 \\&
 \frac{1}{2\lambda^2}| \nabla \psi |^2= \frac12 \frac{1}{4\lambda^2} 
 \frac{1}{\mathtt{m}+\rho} |\nabla \rho|^2+\frac12 (\mathtt{m}+\rho) |\nabla \phi|^2\,.
 \end{align}
 Moreover, using \eqref{EK3}, \eqref{alberobello}, \eqref{mad}
 and that
 \[
 {\rm div}(\rho\nabla\phi)=\nabla\rho \cdot\nabla\phi+\rho\Delta\phi\,,
 \]
 we get
 \begin{equation}\label{papsipsi}
 \begin{aligned}
 \pa_{t}\psi&=\ii e^{\ii\lambda\phi}\Big(\frac{\Delta\rho}{4\lambda\sqrt{\mathtt{m+\rho}}}-\frac{|\nabla\rho|^{2}}{8\lambda(\mathtt{m}+\rho)^{\frac{3}{2}}}
 +\frac{\ii\sqrt{\mathtt{m}+\rho}\Delta\phi}{2}
 -\frac{\sqrt{\mathtt{m+\rho}}\lambda|\nabla\phi|^{2}}{2}+\frac{\ii \nabla\rho\cdot\nabla\phi}{2\sqrt{\mathtt{m}+\rho}}
 \Big)\\&-\ii\lambda g(|\psi|^{2})\psi\,.
 \end{aligned}
 \end{equation}
 On the other hand, by \eqref{winston}, we have
 \begin{equation}\label{papsipsi2}
 \ii\Delta\psi=\ii e^{\lambda\phi}\Big(
 \frac{\Delta\rho}{2\sqrt{\mathtt{m}+\rho}}-\frac{|\nabla\rho|^{2}}{4(\mathtt{m}+\rho)^{\frac{3}{2}}}
 +\ii\lambda \sqrt{\mathtt{m}+\rho}\Delta\phi-\lambda^{2}\sqrt{\mathtt{m}+\rho}|\nabla\phi|^{2}
 +\frac{\ii\lambda\nabla\rho\cdot\nabla\phi}{\sqrt{\mathtt{m}+\rho}}
 \Big)\,.
 \end{equation}
 Therefore, by \eqref{papsipsi}, \eqref{papsipsi2}, we deduce
 \begin{equation}\label{papsipsi3}
 \pa_{t}\psi=\frac{\ii}{2\lambda}\Delta\psi-\ii\lambda g(|\psi|^{2})\psi\,,
 \quad {\rm where}\quad \frac{1}{\lambda}=\hbar\,,
 \end{equation}
 which is the \eqref{papsipsi4}.
 \end{proof}
 Notice that the \eqref{papsipsi4} is an Hamiltonian equation of the form
  \begin{equation}\label{papsipsi5}
 \pa_{t}\psi=-\ii \pa_{\bar \psi}\mathcal{H}(\psi,\bar{\psi})\,,\qquad
  \mathcal{H}(\psi,\bar{\psi})=
 \int_{\mathbb{T}^{d}_{\nu}}
 \big(\frac{\hbar}{2}|\nabla\psi|^{2}+\frac{1}{\hbar}G(|\psi|^{2})\big)dx\,,
 \end{equation}
 where $\pa_{\bar{\psi}}=(\pa_{\Re \psi}+\ii \pa_{\Im \psi})/2$.
  The Poisson bracket is
   \begin{equation}\label{Poisson}
\{\mathcal{H}, \mathcal{G}\}:=-
\ii \int_{\mathbb{T}^{d}_{\nu}} \pa_{\psi}\mathcal{H}\pa_{\bar{\psi}}\mathcal{G}
-  \pa_{\bar{\psi}}\mathcal{H}\pa_{{\psi}}\mathcal{G} dx\,.
\end{equation}

\subsection{Elimination of the zero mode}\label{eliminazero} 
In the following it would be convenient 
to rescale the space variables 
$x\in \mathbb{T}^{d}_{\nu} \rightsquigarrow \nu\cdot x$ with 
$x\in \mathbb{T}^{d}$
and work with  functions 
belonging to the Sobolev space
$H^{s}(\mathbb{T}^d):=H^{s}(\mathbb{T}_{1}^d)$, i.e. the Sobolev space
in \eqref{spazioSob} with $\nu=(1,\ldots,1)$.
By  using the notation 
$ \psi= {(2\pi)^{-\frac{d}{2}}}\sum_{j\in \Z^d} \psi_j e^{\ii j\cdot x},$
 we introduce the set of variables 
\be \label{faou}
\begin{cases}
 \psi_0= \alpha e^{-\ii \theta} & \alpha \in [0,+\infty)\,,\, \theta \in \T \\
 \psi_j=  z_j e^{-\ii \theta} & j\not=0\,,
\end{cases}
\ee 
which are the polar coordinates for $j=0$ and a phase translation for $j\not=0$. 
Rewriting \eqref{papsipsi5} in Fourier coordinates one has 
\be 
\ii\pa_t \psi_j = \pa_{\bar{\psi_j}}\mathcal{H}
(\psi,\bar \psi)\,, \quad j\in \Z^d\,,
\ee 
where $\mathcal{H}$ is defined in \eqref{papsipsi5}. We define also the zero mean variable 
\begin{equation}\label{zeta}
z:= {(2\pi)^{-\frac{d}{2}}}\sum_{j\in \Z^d\setminus\{ 0\} } z_j e^{\ii j\cdot x}\,.
\end{equation}
By \eqref{faou} and \eqref{zeta} one has 
\be\label{faouinv}
\psi= ( \alpha + z) e^{\ii\theta}\,,
\ee
and it is easy to prove that the quantity
\[
 \mathtt{m}:= \sum_{j\in \Z^d} |  \psi_j|^2= \alpha^2 +  \sum_{j\not=0} | z_j|^2
\]
 is a constant of motion for \eqref{papsipsi4}. Using \eqref{faou}, 
 one can completely recover the real variable $\alpha$ 
 in terms of $\{ z_j\}_{j\in \Z^d \setminus \{0\}}$ as 
\begin{equation}\label{def:alpha}
\alpha= \sqrt{ \mathtt{m}- \sum_{j\not=0} |  z_j|^2}\,.
\end{equation}
Note also  that the $(\rho,\phi)$ variables in \eqref{madelunginv} 
do not depend on the angular variable $\theta$ 
defined above. This implies that system \eqref{EK3} is 
completely described by the complex variable $z$.
On the other hand, using 
\[ 
\pa_{\bar{\psi_j}}\mathcal{H}(\psi e^{\ii \theta},\bar {\psi e^{\ii \theta}})
= 
\pa_{\bar{\psi_j}}\mathcal{H}(\psi,\bar \psi)e^{\ii \theta}\,,
\] 
one obtains
\be \label{equationnew}
\begin{cases}
\ii \pa_t \alpha+\pa_t\theta \alpha = \Pi_0\left( g( |\alpha+z|^2) (\alpha + z) \right)  
\\
\ii \pa_t  z_j + \pa_t\theta  z_j=   \frac{\pa \mathcal{H}}{\pa \bar \psi_j}(\alpha+ z ,\alpha +\bar z)\,.
\end{cases}
\ee 
Taking the real part of the first equation in \eqref{equationnew} we obtain 
\be \label{deteta}
 \pa_t\theta= \frac{1}{\alpha}  
 \Pi_0\left( \frac{1}{\hbar}g( |\alpha+z|^2)\Re  (\alpha + z) \right) 
 = \frac{1}{2\alpha} 
 \pa_{\bar{\alpha}}\mathcal{H}(\alpha, z, \bar z)\,,
\ee
where  (recall \eqref{gatto25})
\be 
\tilde{\mathcal{H}}(\alpha, z, \bar z):= 
\frac{\hbar}{2}\int_{\T^d} |D|_{\nu}^{2}z\cdot\bar{z}{\rm d}x
+ \frac{1}{\hbar}\int_{\T^d} G(| \alpha+z|^2)\, {\rm d}x\,.
\ee 
By \eqref{deteta}, \eqref{equationnew} and using that  
\[ 
\pa_{\bar{\psi_j}}\mathcal{H}(\alpha+ z ,\alpha +\bar z)
=
\pa_{\bar{z_{j}}}\tilde{\mathcal{H}}(\alpha, z, \bar z)\,,
\] 
one obtains 
\be \label{fresca}
\begin{aligned}
\ii \pa_t  z_j =  & 
\pa_{\bar{z_j}}\tilde{\mathcal{H}}(\alpha, z, \bar z)
- \frac{z_j}{2\alpha}
\pa_{\alpha}\tilde{\mathcal{H}}(\alpha, z, \bar z)=
 \pa_{\bar{z_j}}\mathcal{K}_{\mathtt{m}}( z, \bar z)\,, \quad j\not=0\,, 
\end{aligned}
\ee
where 
\[
\mathcal{K}_{\mathtt{m}}(z,\bar z):= 
\tilde{\mathcal{H}}(\alpha,z,\bar z)_{|\alpha=\sqrt{\mathtt{m}- \sum_{j\not=0} | z_j|^2}}\,.
\]
We resume the above discussion in the following lemma.
\begin{lemma}\label{federico2}
The following holds. 

 $(i)$ Let $s>\frac{d}{2}$ and 
  \[
  \delta:=\frac{1}{\mathtt{m}}\|\rho\|_{H^{s}}+\frac{1}{\sqrt{\kappa}}\|\Pi_0^{\bot}\phi\|_{H^{s}} \,,
  \quad \theta:=\Pi_0\phi\,.
  \]
  There is $C=C(s)>1$ such that, if $C(s)\delta\leq1$, then the function $z$ in \eqref{zeta} satisfies
  \begin{equation}\label{stimamandzeta}
  \|z\|_{H^{s}}\leq 2\sqrt{\mathtt{m}}\delta\,.
  \end{equation}
  
  \noindent
  $(ii)$ Define
  \[
  \delta':=\|z\|_{H^{s}}\,.
  \]
  There is $C'=C'(s)>1$ such that, if $C'(s) \delta'(\sqrt{\mathtt{m}})^{-1}\leq 1$, then the functions $\rho,$
\begin{equation}\label{stimamand2zeta}
\frac{1}{\mathtt{m}}\|\rho\|_{H^{s}}+\frac{1}{\sqrt{\kappa}}\|\Pi_0^{\bot}\phi\|_{H^{s}}
\leq 16\frac{1}{\sqrt{\mathtt{m}}}\delta'\,.
\end{equation}
 
 \noindent $(iii)$ Let $(\rho,\phi)\in H^s_0(\T_\nu^d)\times H^s(\T_\nu^d)$ 
 be a solution of \eqref{EK3} defined over a time interval $[0,T]$, $T>0$,
  such that
  \begin{equation*}
  \sup_{t\in[0,T)}\Big(\frac{1}{\mathtt{m}}\|\rho(t,\cdot)\|_{H^{s}} +\frac{1}{\sqrt{\kappa}}
\| \Pi_0^{\bot}\phi(t,\cdot)\|_{H^{s}}    \Big)\leq \eps
  \end{equation*}
  for some $\e>0$ small enough. Then the  function $z\in H^s_0(\T^d_\nu)$ 
  defined in \eqref{zeta} solves \eqref{fresca}.

\end{lemma}
\begin{proof}
We note that 
\be 
\| z\|_{H^s}= \| \Pi_0^{\bot} \psi\|_{H^s}\leq \|\psi - \sqrt{m}e^{\ii \theta}\|_{H_s} \stackrel{\eqref{stimamand}}\leq2\sqrt{\mathtt{m}}\delta\,,
\ee 
which proves \eqref{stimamandzeta}. 
In order to prove \eqref{stimamand2zeta} we note that 
\begin{equation*}
\begin{aligned}
\inf_{\sigma\in \T} \| \psi-\sqrt{\mathtt{m}}e^{\ii\sigma}\|_{H^s}&\leq 
\| \psi- \sqrt{\mathtt{m}} e^{\ii \theta}\|_{H^s}
=\| \alpha-\sqrt{\mathtt{m}}+ z\|_{H^s} 
\\
&\leq \sqrt{\mathtt{m}- \| z\|_{L^2}^2}-\sqrt{\mathtt{m}}+ \| z\|_{H^s}\leq 2\delta'\,,
\end{aligned}
\end{equation*}
so that the  \eqref{stimamand2zeta} follows by \eqref{stimamand2}. 
The point $(iii)$ follows by \eqref{equationnew} and \eqref{deteta}.
\end{proof}
\begin{remark}
Using \eqref{madelunginv} and \eqref{faouinv} one can study the system \eqref{EK3} 
near the equilibrium point $(\rho,\phi)=(0,0)$ by studying the complex hamiltonian system 
\be\label{zetaequation}
\ii \pa_t z = \pa_{\bar{z}}\mathcal{K}_{\mathtt{m}}(z,\bar{z})
\ee 
near the equilibrium $z=0$. Note also that the natural phase-space for \eqref{zetaequation} 
is the complex Sobolev space $H_0^s(\T^d; \C )$, $s\in \R$, 
of complex Sobolev functions with zero mean.
\end{remark}

\subsection{Taylor expansion of the Hamiltonian} 
In order to study the stability of $z=0$ for \eqref{zetaequation} it is useful to 
expand $\mathcal{K}_{\mathtt{m}}$ at $z=0$. 
We have
\be\label{espansione}
\begin{aligned}
\mathcal{K}_{\mathtt{m}}(z,\bar z)&= 
\frac{\hbar}{2} \int_{\T^d} |D|_{\nu}^{2}z\cdot\bar{z}\, {\rm d} x 
+ \frac{1}{\hbar} \int_{\T^d} G\Big(\big| \sqrt{\mathtt{m}- \sum_{j\not=0} | z_j|^2}+ z\big|^2\Big)\, {\rm d}x 
\\&
=(2\pi)^d \frac{G(\mathtt{m})}{\hbar}
+\mathcal{K}_{\mathtt{m}}^{(2)}(z,\bar z)
+ \sum_{r= 3}^{N-1} \mathcal{K}_{\mathtt{m}}^{(r)}(z,\bar z)+R^{(N)}(z,\bar z)\,,
\end{aligned}
\ee
where
\be 
\mathcal{K}_{\mathtt{m}}^{(2)}(z,\bar z)=
\frac12 \int_{\T^d} \frac{\hbar}{2}|D|_{\nu}^{2} z\cdot\bar{z}\, {\rm d} x
+ \frac{g'(\mathtt{m})\mathtt{m}}{\hbar}\int_{\T^d} \frac12( z+\bar z)^2\, {\rm d}x\,,
\ee 
for any $r=3, \cdots, N-1$, 
$\mathcal{K}_{{\rm m}}^{(r)}(z,\bar z)$ 
 is an homogeneous multilinear Hamiltonian function of degree $r$ of the form 
 \begin{equation*}
 \mathcal{K}_{\mathtt{m}}^{(r)}(z,\bar z)=
 \sum_{\substack{\sigma\in\{-1,1\}^r,\ j\in(\Z^d\setminus\{0\})^r
 \\ 
\sum_{i=1}^r\sigma_i j_i=0}}
(\mathcal{K}_{\mathtt{m}}^{(r)})_{\sigma,j}
z_{j_1}^{\sigma_1}\cdots z_{j_r}^{\sigma_r}\,,
\qquad 
|(\mathcal{K}_{\mathtt{m}}^{(r)})_{\sigma,j}|\lesssim_{r}1\,,
\end{equation*}
and
\begin{equation}\label{R8}
\| X_{R^{(N)}}(z)\|_{H^{s}}\lesssim_s \| z\|_{H^s}^{r-1}\,,
\qquad
\forall\, z\in B_{1}(
H_0^{s}(\mathbb{T}^{d};\mathbb{C})) \,.
\end{equation}
The vector field of the Hamiltonian in \eqref{espansione} has the form
(recall \eqref{EK2Ham})
\be \label{linear}
\begin{aligned}
\pa_t\begin{bmatrix} z \\ \bar z\end{bmatrix}= 
\begin{bmatrix}-\ii \pa_{\bar{z}}  \mathcal{K}_{\mathtt{m}}\\ 
\ii\pa_{ z} \mathcal{K}_{\mathtt{m}}\end{bmatrix}
= 
&-\ii \begin{pmatrix} \frac{\hbar| D|_{\nu}^2}{2} 
+ \frac{\mathtt{m} g'(\mathtt{m})}{\hbar}&  
\frac{\mathtt{m}g'(\mathtt{m})}{\hbar}\\ 
-\frac{\mathtt{m} g'(\mathtt{m})}{\hbar}
 &-\frac{\hbar| D|_{\nu}^2}{2} 
 - \frac{\mathtt{m} g'(\mathtt{m})}{\hbar}\end{pmatrix}
 \begin{bmatrix} z \\ \bar z \end{bmatrix}
 \\&+
 \sum_{r=3}^{N-1}\begin{bmatrix} -\ii\pa_{\bar{z}}  \mathcal{K}_{\mathtt{m}}^{(r)}\\ 
\ii \pa_{ z} \mathcal{K}_{\mathtt{m}}^{(r)}\end{bmatrix}+
\begin{bmatrix} -\ii\pa_{\bar{z}}  R^{(N)}\\ 
\ii\pa_{ z} R^{(N)}\end{bmatrix}\,.
\end{aligned}
 \ee
 Let us now introduce the 
  $2\times2$ matrix of operators
 \[
 \mathcal{C}:=\frac{1}{\sqrt{2\omega(D)
A(D,\mathtt{m})}}
 \left(
 \begin{matrix}
A(D,\mathtt{m}) & 
 -\tfrac{1}{2}\mathtt{m}g'(\mathtt{m})
 \\
 -\tfrac{1}{2}\mathtt{m}g'(\mathtt{m}) & A(D,\mathtt{m}) \end{matrix}
 \right)\,,
 \]
 with 
  \[
A(D,\mathtt{m}):= \omega(D)
 +\tfrac{\hbar}{2}|D|_{\nu}^{2}+\tfrac{1}{2}\mathtt{m}g'(\mathtt{m})\,,
 \]
 and where $\omega(D)$ is the
 Fourier multiplier  with symbol 
 \begin{equation}\label{simboOmega}
 \omega(j):=\sqrt{ \frac{\hbar^2}{4} |j|_{\nu}^4+ \mathtt{m}g'(\mathtt{m}) |j|_{\nu}^2}\,.
 \end{equation}
 Notice that, by using \eqref{elliptic}, the matrix $\mathcal{C}$ is bounded, invertible and symplectic, with estimates
 \begin{equation}\label{stimeC}
 \|\mathcal{C}^{\pm1}\|_{\mathcal{L}{(H^s_0\times H^s_0,\,\, H^s_0\times H^s_0)}}\leq1+\sqrt{k}\beta,\quad \beta:=\frac{\mathtt{m}g'(\mathtt{m})}{k}.
 \end{equation}
 Consider the change of variables 
 \begin{equation}\label{def:WWW}
\begin{bmatrix} w \\ \bar w \end{bmatrix} := 
\mathcal{C}^{-1} \begin{bmatrix} z \\ \bar z \end{bmatrix}\,.
\end{equation}
then the Hamiltonian \eqref{espansione} reads
\begin{equation}\label{HamKK}
\begin{aligned}
&\widetilde{\mathcal{K}}_{\mathtt{m}}(w,\bar{w}):=
\widetilde{\mathcal{K}}^{(2)}_{\mathtt{m}}(w,\bar{w})
+\widetilde{\mathcal{K}}_{\mathtt{m}}^{(3)}(w,\bar{w})+
\widetilde{\mathcal{K}}^{(\geq4)}_{\mathtt{m}}(w,\bar{w})\,,
\\&
\widetilde{\mathcal{K}}^{(2)}_{\mathtt{m}}(w,\bar{w}):=
\mathcal{K}^{(2)}_{\mathtt{m}}\Big(\mathcal{C}
\begin{bmatrix} w \\ \bar w \end{bmatrix}\Big):=\frac{1}{2}
\int_{\mathbb{T}^{d}}\omega(D)z\cdot\bar{z}{\rm d}x\,,
\\&
\widetilde{\mathcal{K}}^{(3)}_{\mathtt{m}}(w,\bar{w}):=
\mathcal{K}^{(3)}_{\mathtt{m}}\Big(\mathcal{C}
\begin{bmatrix} w \\ \bar w \end{bmatrix}\Big)\,,
\qquad 
\widetilde{\mathcal{K}}_{\mathtt{m}}^{(\geq 4)} (w,\bar w):= 
\sum_{r=4}^{N-1} \mathcal{K}_{\mathtt{m}}^{(r)}\Big(\mathcal{C}
\begin{bmatrix} w \\ \bar w \end{bmatrix}\Big)
+ R^{(N)}\Big(\mathcal{C}
\begin{bmatrix} w \\ \bar w \end{bmatrix}\Big)\,.
\end{aligned}
\end{equation}
Therefore system \eqref{linear} becomes
\begin{equation}\label{linear22}
\pa_{t}w=-\ii \omega(D)w-\ii \pa_{\bar{w}}\widetilde{\mathcal{K}}^{(3)}_{\mathtt{m}}(w,\bar{w})
-\ii \pa_{\bar{w}}\widetilde{\mathcal{K}}^{(\geq4)}_{\mathtt{m}}(w,\bar{w})\,.
\end{equation}

\section{Small divisors}\label{sec:meas}
As explained in the introduction  we shall study the long time behaviour of 
solutions of \eqref{linear22} by means of Birkhoff normal form approach.
Therefore we have to provide suitable non resonance conditions 
among linear frequencies of oscillations $\omega(j)$ in \eqref{simboOmega}.
This is actually the aim of this section.
 
Let $\ta=(\ta_1,\ldots,\ta_d)=(\nu_1^2,\ldots,\nu_d^2) \in (1,4)^d$, $d=2,3$. 
If $j\in\Z^d\setminus\{0\}$ we define
\begin{equation}
|j|^2_{\ta}=\sum_{k=1}^d\ta_k j_{k}^2\,.
\end{equation}
We consider the dispersion relation 
\begin{equation}\label{disp-rel}
\omega(j):=\sqrt{k|j|^4_{\ta}+\mathtt{m}g'(\mathtt{m})|j|^2_{\ta}}\,,
\end{equation}
we note that 
$\omega(j)=\sqrt{k}\big(|j|_{\ta}^2+\frac{\beta}{2}-\frac{\beta^2}{8}\frac{1}{|j|_{\ta}^2}
+O(\frac{\beta^3}{|j|_{\ta}^4})\big)$ 
for any $j$ big enough with respect to $\beta:=\frac{\mathtt{m}g'(\mathtt{m})}{k}$.\\

Throughout this section we assume, without loss of generality, 
$|j_1|_{\ta}\geq |j_2|_{\ta}\geq |j_3|_{\ta}>0$, 
for any $j_i$ in $\Z^d$, moreover,
in order to lighten the notation,  
we  adopt the convention $\omega_i:=\omega(j_i)$ for any $i=1,2,3$.
The main result is the following.
\begin{proposition}{\bf (Measure estimates).}\label{measures}
There exists a full Lebesgue measure set $\mathfrak{A}\subset(1,4)^d$ 
such that for any $\ta\in\mathfrak{A}$ there exists $\gamma>0$ 
such that the following  holds true. 
If $\sigma_1j_1+\sigma_2j_2+\sigma_3j_3=0$, $\sigma_i\in\{\pm 1\}$ 
we have the estimate
\begin{equation}\label{misuratot}k^{-\frac12}\big|\sigma_3\omega_{3}+\sigma_2\omega_{2}+\sigma_1\omega_{1}\big|\gtrsim_d 
\begin{cases}
&\frac{\gamma}{|j_1|^{d-1}\log^{\,d+1}{(1+|j_1|^2)}|j_3|^{M(d)}}\,, \quad \mbox{\,if\,} \sigma_{1}\sigma_2=-1\\
& 1,\, \quad\quad\quad\qquad\mbox{if\,} \sigma_{1}\sigma_2=1
\end{cases}.
\end{equation}
for any $|j_1|_{\ta}\geq|j_{2}|_{\ta}\geq|j_3|_{\ta}$, $j_i\in\Z^d$ and where 
$M(d)$ is a constant depending only on $d$.
\end{proposition}

The proof of this proposition is divided in several steps 
and it is postponed to the end of the section. 
The main ingredient is the following standard proposition 
which follows the lines of \cite{Bambu,BFGI}. 
Here we give weak lower bounds of the small divisors, 
these estimates will be improved later.

\begin{proposition}\label{dario}
Consider $I$ and $J$  two bounded intervals of 
$\,\R^+\setminus{\{0\}}$; $r\geq 2$ and $j_1,\ldots,j_r\in \Z^d$ such that
$j_i\neq\pm j_k$ if $i\neq k$, $n_1,\ldots n_r\in \Z\setminus{\{0\}}$ and 
$h:J^{d-1}\rightarrow \R$ measurable. Then for any $\gamma>0$ we have
\begin{equation*}
\mu\Big\{(\mathtt{p},\mathtt{b})\in I\times J^{d-1}: \big|h(\mathtt{b})
+\sum_{k=1}^rn_k\sqrt{|j_k|^4_{(1,\mathtt{b})}+\mathtt{p}|j_k|^2_{(1,\mathtt{b})}}\big|
\leq\gamma\Big\}\lesssim_{I,J,d,r,n} 
\gamma^{\frac{1}{2r}} (\langle j_1\rangle \cdots \langle j_r\rangle)^{\frac1r}\,,
\end{equation*}
with $(1,\mathtt{b})=(1,\mathtt{b}_1,\ldots,\mathtt{b}_{d-1})\in \R^d$.
\end{proposition}
\begin{remark}
We shall apply this general proposition only in the case $r=3$, 
however we preferred to write it in general for possible future applications.
\end{remark}
\begin{proof}[{P}roof of Prop. \ref{dario}]
For simplicity in the proof we assume 
$|j_1|_{(1,\mathtt{b})}>\ldots> |j_r|_{(1,\mathtt{b})}$.
Since by assumption we have $j_i\neq j_k$ for any $i\neq k$ 
then one could easily prove that for any $\eta>0$ 
(later it will be chosen in function of $\gamma$) we have
\begin{equation*}
\mu(P_{\eta}^{i,k})<\eta \mu(J^{d-2})\,, 
\quad P_{\eta}^{i,k}:=\{\mathtt{b}\in J^{d-1}:|j_i|^2_{(1,\mathtt{b})}-|j_k|^2_{(1,\mathtt{b})}<\eta\}\,.
\end{equation*}
We define $P_{\eta}=\cup_{i\neq k}P^{i,k}_{\eta}$, and
\[
B_{\gamma}:=\Big\{(\mathtt{p},\mathtt{b})\in I\times J^{d-1}: 
\big|h(\mathtt{b})+\sum_{k=1}^rn_k\sqrt{|j_k|^4_{(1,\mathtt{b})}
+\mathtt{p}|j_k|^2_{(1,\mathtt{b})}}\big|\leq\gamma\Big\}\,,
\] 
 then we have
 \begin{equation*}
 \begin{aligned}
 \mu(B_{\gamma})&\leq \mu(B_{\gamma}\cap P_{\eta})+\mu(B_{\gamma}\cap (P_{\eta})^c)
 \\
& \leq \mu(I)\mu(P_{\eta})	+\mu(J^{d-1})
\sup_{\stackrel{i\neq k}{\mathtt{b}\notin P_{\eta}}}\mu\Big(\{\mathtt{p}\in I:\big|h(\mathtt{b})
+\sum_{k=1}^rn_k\sqrt{|j_k|^4_{(1,\mathtt{b})}+\mathtt{p}|j_k|^2_{(1,\mathtt{b})}}\big|\leq\gamma\}\Big)
\\&
\lesssim_r \mu(I)\mu(J^{d-2})\eta
\\&+\mu(J^{d-1})\sup_{\stackrel{i\neq k}{\mathtt{b}\notin P_{\eta}}}
\mu\Big(\{\mathtt{p}\in I:\big|h(\mathtt{b})+\sum_{k=1}^rn_k
\sqrt{|j_k|^4_{(1,\mathtt{b})}+\mathtt{p}|j_k|^2_{(1,\mathtt{b})}}\big|\leq\gamma\}\Big)\,.
\end{aligned} \end{equation*}
We have to estimate from above the measure of the last set. We define the function $$g(\mathtt{p}):=h(\mathtt{b})+\sum_{k=1}^rn_k\sqrt{|j_k|^4_{(1,\mathtt{b})}+\mathtt{p}|j_k|^2_{(1,\mathtt{b})}}.$$ For any $\ell\geq 1$ we have
\begin{equation*}
\frac{d^{\ell}}{d\mathtt{p}^{\ell}}g(\mathtt{p})=c_{\ell}\sum_{k=1}^rn_k|j_k|  _{(1,\mathtt{b})}
(\mathtt{p}+|j_k|^2  _{(1,\mathtt{b})}
)^{\frac12-\ell}, \quad c_{\ell}:=\prod_{i=1}^{\ell}(\tfrac12-i).
\end{equation*}
Therefore we can write the system of equations\[
\left(\begin{matrix} 
c_1^{-1}\partial_{\mathtt{p}}^1g(\mathtt{p})\\
\vdots\\
c_r^{-1}\partial_{\mathtt{p}}^{r}g(\mathtt{p})\end{matrix}\right)=\left(
\begin{matrix}
(\mathtt{p}+|j_1|  _{(1,\mathtt{b})}
^2)^0 & \ldots  & (\mathtt{p}+|j_r|  _{(1,\mathtt{b})}
^2)^0\\
\vdots &  \ddots & \vdots\\
 (\mathtt{p}+|j_1|  _{(1,\mathtt{b})}
^2)^{1-r} &  \ldots &  (\mathtt{p}+|j_r|  _{(1,\mathtt{b})}
^2)^{1-r}
\end{matrix}
\right) 
\left(\begin{matrix}
n_1|j_1|  _{(1,\mathtt{b})}
(\mathtt{p}+|j_1|  _{(1,\mathtt{b})}
^2)^{-1/2}\\
\vdots\\
n_r|j_r|  _{(1,\mathtt{b})}
(\mathtt{p}+|j_r|  _{(1,\mathtt{b})}
^2)^{-1/2}\end{matrix}\right)\,.
\]
We denote by $V$ the Vandermonde matrix above. 
We have that $V$ is invertible since 
\begin{equation*}
\begin{aligned}
|\det(V)|&=\prod_{1\leq i<k\leq r}\big|\frac{1}{\mathtt{p}
+|j_i|_{(1,\mathtt{b})}^2}-\frac{1}{\mathtt{p}
+|j_k|_{(1,\mathtt{b})}^2}\big|\geq \prod_{1\leq i<k\leq r}
\frac{\big||j_i|_{(1,\mathtt{b})}^2-|j_k|_{(1,\mathtt{b})}^2\big|}{(\mathtt{p}
+|j_i|_{(1,\mathtt{b})}^2)(\mathtt{p}+|j_k|_{(1,\mathtt{b})}^2)}
\\&
\gtrsim \prod_{1\leq k\leq r} \frac{\eta}{(\mathtt{p}+|j_k|_{(1,\mathtt{b})}^2)^2}
\gtrsim \eta^{r} \frac{1}{\langle j_1\rangle^2\cdots\langle j_r\rangle^2}\,,
\end{aligned}\end{equation*} 
where in the penultimate passage we have used that 
$\mathtt{b}\notin P_{\eta}$ and $|j_i|_{(1,\mathtt{b})}\leq |j_k|_{(1,\mathtt{b})}$ if $i>k$.
Therefore we have
\begin{equation*}
\begin{aligned}
\max_{\ell=1}^r|c_{\ell}\partial_{\mathtt{p}}^{\ell}g(\mathtt{p})|
&
\gtrsim_r |\det(V)|\max_{\ell=1}^r\Big|n_{\ell}|j_{\ell}|_{(1,\mathtt{b})}
(\mathtt{p}+|j_{\ell}|)_{(1,\mathtt{b})}^{-\frac12}\Big|
\\&
\gtrsim_{r,n} \eta^r \frac{|j_1|_{(1,\mathtt{b})}^{1/2}}{\langle j_1\rangle^2\cdots\langle j_r\rangle^2}
\gtrsim_{r,n} \frac{\eta^r}{\langle j_1\rangle^2\cdots\langle j_r\rangle^2}\,.
\end{aligned}
\end{equation*}
At this point we are ready to use Lemma 7 in appendix A of the paper 
\cite{P-P}, we obtain
\[\mu\Big(\{\mathtt{p}\in I:\big|h(\mathtt{b})
+\sum_{k=1}^r\sqrt{|j_k|^4_{(1,\mathtt{b})}+\mathtt{p}|j_k|^2_{(1,\mathtt{b})}}\big|
\leq\gamma\}\Big)\leq 
(\frac{\gamma\langle j_1\rangle^2\ldots\langle j_r\rangle^2}{\eta^r}\Big)^{\frac{1}{r}}\,.
\]
Summarizing we obtained 
\begin{equation*}
\mu(B_{\gamma})\lesssim_{I,J,d,r,n} \eta
+\eta^{-1}\gamma^{\frac1r}(\langle j_1\rangle^2\ldots\langle j_r\rangle^2)^{\frac1r}\,,
\end{equation*}
we may optimize by choosing 
$\eta=\gamma^{\frac{1}{2r}}(\langle j_1\rangle\cdots\langle j_r\rangle)^\frac1r$ 
and we obtain the thesis.
\end{proof}
As a consequence of the preceding proposition we have the following.
\begin{corollary}\label{corollario}
Let $r\geq 1$, consider $j_1,\ldots,j_r\in\mathbb{Z}^d$ such that $j_k\neq j_i$ if $i\neq k$ and $n_1,\ldots, n_k\in\Z\setminus\{0\}$. For any $\gamma>0$ we have
\begin{equation*}
\mu\Big(\Big\{a\in(1,4):\, \sum_{i=1}^r n_i\sqrt{k|j_i|_{\ta}^4+\mathtt{m}g'(\mathtt{m})|j_i|_{\ta}^2}\leq \gamma \Big\}\Big)\lesssim_{d,r,n}\Big(\frac{\gamma}{\sqrt{k}}\Big)^{\frac{1}{2r}}(\langle j_1\rangle\ldots\langle j_r\rangle)^{\frac1r}.
\end{equation*}
\end{corollary}
\begin{proof}
We write
\begin{equation*}
\sum_{i=1}^r n_i\sqrt{k|j_i|_{\ta}^4+\mathtt{m}g'(\mathtt{m})|j_i|_{\ta}^2}\Big)=\sqrt{k}\ta_1\sum_{i=1}^rn_i\sqrt{|j_i|^4_{(1,\mathtt{b})}+\frac{\beta}{\ta_1}|j_i|^2_{(1,\mathtt{b})}},
\end{equation*}
where we have set
\begin{equation}\label{beta}
\beta:=\frac{\mathtt{m}g'(\mathtt{m})}{k},\quad \mathtt{b}:=(\frac{\ta_2}{\ta_1},\ldots,\frac{\ta_d}{\ta_1}).
\end{equation}
The map $(\ta_1,\dots,\ta_d)\mapsto (\frac{1}{\ta_1},\mathtt{b})$ is invertible onto its image, which is contained in $(\frac{1}{4},1)\times (\frac14,4)^{d-1}$. The determinant of its inverse is bounded by a constant depending only on $d$. Therefore the result follows by applying Prop. \ref{dario} and the change of coordinates $(\ta_1,\dots,\ta_d)\mapsto (\frac{1}{\ta_1},\mathtt{b})$.
\end{proof}

Owing to the corollary above we may reduce in the following 
to the study of the small dividers when we 
have $2$ frequencies much larger then the other.

\begin{lemma}\label{casocaso}
Consider $\tilde{\Lambda}:=\sqrt{k}|j_1|_{\ta}^2-\sqrt{k}|j_2|_{\ta}^2-\omega_3$ and $\beta$ defined in \eqref{beta}. 
If there exists $i\in\{1,\ldots,d\}$ such that 
\begin{equation}\label{cond-grandezza}
|j_{3,i}|\sqrt{1+\frac{\beta}{|j_3|_{\ta}^2}}\leq \frac12 |j_{1,i}+j_{2,i}|\,,
\end{equation}
then for any $\tilde{\gamma}>0$ we have
\begin{equation*}
\mu\Big(\big\{\ta\in(1,4)^d: |\tilde{\Lambda}|\leq \tilde{\gamma}\big\}\Big)
\leq \frac{2\tilde{\gamma}}{\sqrt{k}|j_{1,i}+j_{2,i}|}\,.
\end{equation*}
\end{lemma}
\begin{proof}
We give a lower bound for the derivative of the function 
$\tilde{\Lambda}$ with respect to the parameter $a_i$.
\begin{equation*}
|\partial_{a_i}\tilde{\Lambda}|\geq\sqrt{k}\Big[|j_{3,i}(j_{1,i}+j_{2,i})|-j_{3,i}^2\sqrt{1+\frac{\beta}{|j_3|_a^2}}\Big]\geq\sqrt{k}\frac12 |j_{3,i}||j_{1,i}+j_{2,i}|\geq\sqrt{k}\frac12 |j_{1,i}+j_{2,i}|.
\end{equation*}
Therefore $a_i\mapsto \tilde{\Lambda}$ is a diffeomorphism and 
applying this change of variable we get the thesis.
\end{proof}
\begin{proposition}\label{3frequenze}
There exists   a set of full Lebesgue measure $\mathfrak{A}_3\subset(1,4)^d$ 
such that for any $\ta$ in $\mathfrak{A}_3$ there exists $\gamma>0$  such that 
\begin{equation*}
\big|\sigma\omega_3+\omega_2-\omega_1\big|
\geq\frac{\sqrt{k}\gamma}{|j_1|^{d-1}\log^{\,d+1}{(1+|j_1|^2)}|j_3|^{d+1}}\,,
\end{equation*}
for any $\sigma\in{\pm 1}$, for any $j_1, j_2, j_3$ in $\Z^d$ 
satisfying $|j_1|_{\ta}>|j_2|_{\ta}\geq |j_3|_{\ta}$, the 
momentum condition $\sigma j_3+j_2-j_1=0$ and 
\begin{equation}\label{max>>min}
\mathfrak{J}(j_1,\beta)=\min\left\{\frac{\sqrt{||j_1|^2-4d^2\beta|}}{2d},
\min\Big\{\Big(\frac{\gamma}{4\beta^2}\Big)^{\frac{1}{d+2}},\Big(\frac{\gamma}{2\beta^3}\Big)^{\frac{1}{d+1}}\Big\}\Big(\frac{|j_1|^{4-d}}{\log(1+|j_1|)^{d+1}}\Big)^{\frac{1}{d+2}}\right\}> |j_3|\,,
\end{equation}
where $\beta$ is defined in \eqref{beta}.
\end{proposition}
\begin{proof}
We suppose $\sigma=1$, we set 
$\Lambda:=\omega_1-\omega_2-\omega_3$ and 
$$L(\gamma):=\frac{\sqrt{k}{\gamma}}{{\big(|j_3|^{d+1}|j_1|^{d-1}\log^{d+1}{(1+|j_1|)}\big)}}.$$
From the first condition in \eqref{max>>min} we deduce that $\beta/|j_1|^2<1$, therefore, by Taylor expanding the \eqref{disp-rel},  we obtain
\begin{equation}\label{lambda1}
\Lambda=\sqrt{k}\Big(|j_1|_{\ta}^2-|j_2|_{\ta}^2+\frac{\beta^2}{8}
\frac{|j_1|^2_{\ta}-|j_2|^2_{\ta}}{|j_2|_{\ta}^2|j_1|_{\ta}^2}+R\Big)-\omega_3\,,
\end{equation}
where $|R|\leq \frac18\frac{\beta^3}{|j_2|_{\ta}^4}$.
We define 
$\tilde{\Lambda}:=\sqrt{k}|j_1|_{\ta}^2-\sqrt{k}|j_2|_{\ta}^2-\omega_3$ 
and the following \emph{good sets}
\begin{equation*}
\begin{aligned}
\mathcal{G}_{\gamma}&:=\{\ta\in (1,4)^d:|\Lambda|>L(\gamma), \forall j_1,j_3\in \mathbb{Z}^d\}\,, 
\\
\tilde{\mathcal{G}}_{\gamma}&:=\{\ta\in(1,4)^d:\big|\tilde{\Lambda}\big|>3L(\gamma) \,,
\forall j_1,j_3\in \mathbb{Z}^d\}\,.
\end{aligned}
\end{equation*}
We claim that, thanks to \eqref{max>>min}, we have the inclusion 
$\tilde{\mathcal{G}}_{\gamma}\subset{\mathcal{G}}_{\gamma}$. 
First of all we have
\begin{equation}\label{geppetto}
\big|\Lambda\big|\geq\big|\omega_3+\sqrt{k}|j_2|_{\ta}^2-\sqrt{k}|j_1|_{\ta}^2\big|-\sqrt{k}\frac{\beta^2}{8}
\frac{|j_1|_{\ta}^2-|j_2|_{\ta}^2}{|j_1|_{\ta}^2|j_2|_{\ta}^2}-\sqrt{k}|R|\,.
\end{equation}
From the momentum condition $j_1-j_2=j_3$ and the ordering 
$|j_1|_{\ta}>|j_2|_{\ta}\geq|j_3|_{\ta}$ we have that $|j_1|_{\ta}\leq 2|j_2|_{\ta}$, 
which implies
\begin{equation}\label{geppetto2}
\frac{|j_1|_{\ta}^2-|j_2|_{\ta}^2}{|j_1|_{\ta}^2|j_2|_{\ta}^2}=
\frac{\sum_{k=1}^d a_k j_{3,k}(j_{1,k}+j_{2,k})}{|j_1|_{\ta}^2|j_2|_{\ta}^2}
\leq 2\frac{|j_3|_{\ta}|j_1|_{\ta}}{|j_1|_{\ta}^2|j_2|^2_{\ta}} \leq 2 \frac{|j_3|_{\ta}}{|j_1|_{\ta}|j_2|^2_{\ta}}
\leq 32\frac{|j_3|}{|j_1|^3}\,,
\end{equation}
where we used $|\cdot|<|\cdot|_{\ta}<4|\cdot|$. 
Therefore from $\eqref{max>>min}$, more precisely from
\begin{equation*}
(\frac{\gamma}{4\beta^2})^{\frac{1}{d+2}}\Big(\frac{|j_1|^{4-d}}{\log(1+|j_1|)^{d+1}}\Big)^{\frac{1}{d+2}}> |j_3|,
\end{equation*}
we deduce that 
\begin{equation*}\sqrt{k}\frac{\beta^2}{8}
\frac{|j_1|_{\ta}^2-|j_2|_{\ta}^2}{|j_1|_{\ta}^2|j_2|_{\ta}^2}<L.\end{equation*}
Analogously one proves that $\sqrt{k}|R|<L$. We have eventually proved that $\tilde{\mathcal{G}}_{\gamma}\subset{\mathcal{G}}_{\gamma}$ using \eqref{geppetto}.

We define the \emph{bad sets} 
$\tilde{\mathcal{B}}_{\gamma}:=
\big((1,4)^d\setminus\tilde{\mathcal{G}}_{\gamma}\big)\supset 
{\mathcal{B}}_{\gamma}:=\big((1,4)^d\setminus\mathcal{G}_{\gamma}\big)$ 
and we prove that the Lebesgue measure of 
$\cap_{\gamma}\tilde{\mathcal{B}}_{\gamma}$ 
equals to zero, 
this implies the thesis. 

We want to apply Lemma \ref{casocaso} with 
$\tilde{\gamma}\rightsquigarrow L$. We know that there exists $i\in\{1,\ldots,d\}$ such that $d|j_{1,i}|\geq |j_1|$.
We claim that, thanks to \eqref{max>>min},  we satisfy condition \eqref{cond-grandezza} for the same index $i$.
Let us suppose by contradiction that 
\begin{equation*}
|j_{3,i}|\sqrt{1+\tfrac{\beta}{|j_3|^2_{\ta}}}>\tfrac12|j_{1,i}+j_{2,i}|
=\tfrac12|2j_{1,i}-j_{3,1}|\geq|j_{1,i}|-\tfrac12|j_{3,i}|>\tfrac{|j_1|}{d}-\tfrac12 |j_{3,i}|\,,
\end{equation*}
from which we obtain $|j_1|\leq2d|j_{3}|\sqrt{1+\beta/|j_3|^2_{\ta}}$.
Taking the squares  we get
\begin{equation*}
|j_1|^2\leq 4d^2|j_3|^2+4d^2\beta \frac{|j_3|^2}{|j_3|^2_{\ta}}\,,
\end{equation*}
which, recalling that $|\cdot|<|\cdot|_{\ta}<4|\cdot|$, 
contradicts \eqref{max>>min}.

Therefore, by using Lemma \ref{casocaso}, we have
\begin{equation*}
\begin{aligned}
\m\big(\tilde{\mathcal{B}_{\gamma}}\big)&=
\m\big(\big\{\mathtt{a}\in(1,4)^d| \exists j_1,j_3\in\Z^d: 
|\tilde{\Lambda}|\leq \sqrt{k}\gamma |j_3|^{-d-1}|j_1|^{1-d}\log(|j_1|)^{-d-1}\big\}\big)
\\&
\leq\sum_{j_3\in\mathbb{Z}^d}\frac{1}{|j_3|^{d+1}} \sum_{j_1\in\mathbb{Z}^d}
\frac{\gamma}{|j_1|^{d-1}|j_{1,i}|\log(|j_1|)^{d+1}}\lesssim_d\gamma\,.
\end{aligned}
\end{equation*}
This implies that $meas(\cap_{\gamma}\mathcal{B}_{\gamma})=0$, 
hence we can set $\mathfrak{A}_3=\cup_{\gamma}\mathcal{G}_{\gamma}$.
\end{proof}

\def\cprime{$'$}
We are now in position to prove Prop. \ref{measures}.
\begin{proof}[Proof of Prop. \ref{measures}]
The case $\sigma_1\sigma_2=1$ is trivial, we give the proof if $\sigma_1\sigma_2=-1.$  From Prop. \ref{3frequenze} we know that there exists a 
full Lebesgue measure set $\mathfrak{A}_3$  and $\gamma>0$ such that the statement is proven 
if $|j_3|\leq \mathfrak{J}(j_1,\gamma)$. Let us now assume  $|j_3|>\mathfrak{J}(j_1,\gamma)$.
Let us define
\begin{equation*}
\mathcal{B}_{{\gamma}}:=\bigcup_{j_1,j_3\in\Z^d}\Big\{\ta\in(1,4)^d: 
|\sigma_3\omega_3+\omega_2-\omega_1|\leq
\sqrt{k}\frac{{\tilde{\gamma}}}{|j_3|^{M(d)}}\Big\}\,,
\end{equation*}
where $\tilde{\gamma}$ will be chosen in function of $\gamma$ and $M(d)$ big enough w.r.t. $d$.\\
Let us set $p:=(\frac{M(d)}{6}-d-1)\frac{1}{d+2}$ 
suppose for the moment 
$\big(\gamma/4\beta^2\big)^{\frac{1}{d+2}}\leq \big(\gamma/2\beta^3\big)^{\frac{1}{d+1}}$. 
From  $|j_3|>\mathfrak{J}(j_1,\beta)$ (see \eqref{max>>min}) 
and Corollary \ref{corollario} with $r=3$, we have
\begin{equation*}
\begin{aligned}
\mu(\mathcal{B}_{{\gamma}})&\lesssim_{d} (\sqrt{k})^{-\frac16} \sum_{j_1,j_3\in\Z^d} 
\frac{\tilde{\gamma}^{\frac16}}{|j_3|^{M(d)/6}}
\langle j_1\rangle
\\&
\lesssim_{d}(\sqrt{k})^{-\frac16} \tilde{\gamma}^{1/6}\gamma^{-p}(4\beta^2)^p
\sum_{j_1\in\Z^d}\frac{\log^{p(d+1)}(1+|j_1|)}{|j_1|^{(4-d)p-1}}\sum_{j_3}|j_3|^{-d-1}\,.
\end{aligned}
\end{equation*}
If the exponent $M(d)$ (and hence $p$) is chosen large enough we get the 
summability in the r.h.s. of the inequality above. We now choose $ \tilde{\gamma}^{1/6}\gamma^{-p}=\gamma^{m}$,
we eventually obtain
$\mu(\mathcal{B}_{\gamma})\lesssim \gamma^{m}$. If $\big(\gamma/4\beta^2\big)^{\frac{1}{d+2}}> \big(\gamma/2\beta^3\big)^{\frac{1}{d+1}}$ one can reason similarly.
The wanted set of full Lebesgue measure is therefore obtained 
by choosing 
$\mathfrak{A}:=\mathfrak{A}_3\cap(\cup_{\gamma>0} 
\mathcal{B}_{\gamma}^c)$.
\end{proof}

\section{Energy estimates}
In this section we construct a modified energy for the  Hamiltonian 
$\widetilde{\mathcal{K}}_{\mathtt{m}}$ in \eqref{HamKK}.
We first need some convenient notation.
\begin{definition}\label{def:mumumu}
 If $j \in (\mathbb{Z}^d)^r$ for some 
$r\geq k$ then $\mu_k(j)$ denotes the 
$k^{st}$ largest number among $|j_1|, \dots, |j_r|$ 
(multiplicities being taken into account). 
\end{definition}
\begin{definition}{\bf (Formal Hamiltonians).}\label{Ham:class}
We denote by $ \mathcal{L}_3$ 
the set of Hamiltonian  having homogeneity 
$3$ and such that they may be written in the form
\begin{align}
G_{3}(w)&= 
\sum_{\substack{\sigma_i\in\{-1,1\},\ j_i\in\Z^d\setminus\{0\}\\
\s_1j_1+\s_2j_2+\s_3j_3=0
}}
(G_{3})_{\sigma,j}w_{j_1}^{\sigma_1}w_{j_2}^{\sigma_2} w_{j_3}^{\sigma_3}\,,
\quad 
(G_{3})_{\sigma,j}\in \mathbb{C}\,, 
\quad \begin{array}{cl}&\sigma:=(\sigma_1,\s_2,\sigma_3)\\
&j:=(j_1,j_2,j_3)
\end{array}\label{HamG} 
\end{align}
 with symmetric coefficients $(G_3)_{\s,j}$ (i.e.
for any $\rho\in\mathfrak{S}_{3}$
one has $(G_{3})_{\sigma,j}=(G_{3})_{\sigma\circ\rho,j\circ\rho}$) 
and where we denoted
\[
w^{\s}_{j}:=w_{j} \,,\quad {\rm if}\;\;\s=+\,,\qquad 
w^{\s}_{j}:=\ov{w_{j}} \,,\quad {\rm if}\;\;\s=-\,.
\]
\end{definition}
\noindent
The Hamiltonian in \eqref{HamKK} has the form (see \eqref{simboOmega})
\begin{equation}\label{exampleHam}
\widetilde{\mathcal{K}}_{\mathtt{m}}:=\widetilde{\mathcal{K}}^{(2)}_{\mathtt{m}}
+\widetilde{\mathcal{K}}^{(3)}_{\mathtt{m}}+\widetilde{\mathcal{K}}^{(\geq4)}_{\mathtt{m}}\,,\qquad 
\widetilde{\mathcal{K}}^{(2)}_{\mathtt{m}}=
\sum_{j\in \mathbb{Z}^{d}\setminus\{0\}}
\omega(j)w_{j}\bar{w}_{j}\,,
\end{equation}
where $\widetilde{\mathcal{K}}^{(3)}_{\mathtt{m}}$ 
is a trilinear Hamiltonian  in $\mathcal{L}_3$
with coefficients satisfying 
\begin{equation}\label{howimet}
 |(\widetilde{\mathcal{K}}^{(3)}_{\mathtt{m}})_{\sigma,j}|\lesssim 1\,,
 \qquad \forall\; \s\in\{-1,+1\}^{3}\,,\;\; j\in (\mathbb{Z}^{d})^{3}\setminus\{0\}\,,
\end{equation}
and where $\widetilde{\mathcal{K}}^{(\geq4)}_{\mathtt{m}}$ satisfies for any $s>d/2$
\begin{equation}\label{R8888}
\| X_{\widetilde{\mathcal{K}}^{(\geq4)}_{\mathtt{m}}}(w)\|_{H^{s}}\lesssim_s \| w\|_{H^s}^{3}\,,
\qquad
\mbox{if}\,\,\,\,\, \|w\|_{H^s}< 1 \,.
\end{equation}
The main result of this section is the following.
\begin{proposition}\label{modifiedstep}
 Let  $\mathfrak{A}$ and $M$
given by Proposition \ref{measures}. Consider  $\mathtt{a}\in \mathfrak{A}$.
For any  $N>1$
and any $s\geq\tilde{s}_0$, 
for some $\tilde{s}_0=\tilde{s}_0(M)>0$,  
there exist  $\eps_0 \lesssim_{s,\delta} \log^{-d-1} (1+N)$ and a trilinear function $E_{3}$
in the class $\mathcal{L}_{3}$
such that the following holds:

\noindent
$\bullet$ the coefficients 
$(E_{3})_{\sigma,j}$ satisfies
\begin{equation}\label{coeffE6}
|(E_{3})_{\sigma,j}|\lesssim_{s}N^{d-2}\log^{d+1}(1+N)
\mu_{3}(j)^{M+1} \mu_1(j)^{2s}\,,
\end{equation}
for $\sigma\in \{-1,1\}^{3}$, $j\in(\mathbb{Z}^{d})^{3}\setminus\{0\}$;

\noindent
$\bullet$ for any $w$ in the ball of radius 
$\e_0$ of $H_0^{s}(\mathbb{T}^{d};\mathbb{C})$
one has
\begin{equation}\label{energyestimate}
\begin{aligned}
|\{N_{s}+E_3,\widetilde{\mathcal{K}}_{\mathtt{m}}\}|&\lesssim_{s} 
 N^{d-2}\log^{d+1}(1+N)\|w\|_{H^{s}}^{4}
+N^{-1}\|w\|_{H^{s}}^{3}\,.
\end{aligned}
\end{equation}
where  $N_s$ is defined as 
\begin{equation}\label{HamNN}
N_{s}(w):=\|w\|_{H^{s}}^{2}=\sum_{j\in\mathbb{Z}^{d}}\langle j\rangle^{2s}|w_{j}|^{2}\,,
\end{equation}
and $\widetilde{\mathcal{K}}_{\mathtt{m}}$ in \eqref{exampleHam}.
\end{proposition}

In subsection \ref{sec:trilin} we study some properties of the Hamiltonians in $\mathcal{L}_3$
of Def. \ref{Ham:class}. Then in subsection \ref{sec:modstep} we give the proof of 
Proposition \ref{modifiedstep}.
Finally, in subsection \ref{sec:proofthm}, we conclude the proof
of the main theorem. 

 \subsection{Trilinear  Hamiltonians}\label{sec:trilin}
 We now recall some properties 
 of trilinear Hamiltonians introduced in Definition \ref{Ham:class}.
 We first give some further definitions. 
\begin{definition}\label{Ham:class2}
Let $N\in \mathbb{R}$ with $N\geq1$.

\noindent
$(i)$ If $G_{3}\in \mathcal{L}_{3}$ then $G_{3}^{>N}$ denotes the element of 
$\mathcal{L}_{3}$ defined by
\begin{equation}\label{HamGhigh}
(G_{3}^{>N})_{\sigma,j}:=\left\{
\begin{array}{lll}
&(G_{3})_{\sigma,j}\,,&{\rm if} \;\;\mu_{2}(j)>N\,,\\
&0\,,  & {\rm else}\,.
\end{array}
\right.
\end{equation}
We set
$G^{\leq N}_{3}:=G_{3}-G^{>N}_{3}$. 

\noindent
$(ii)$ We define $G_{3}^{(+1)}\in \mathcal{L}_{3}$ 
by
\[
\begin{aligned}
&(G_{3}^{(+1)})_{\s,j}:=(G_{3})_{\s,j}\,,\;\;{\rm when}\;\;\; 
\exists i,p=1,2,3,\;{\rm s.t.}\;\\
& \mu_{1}(j)=|j_{i}|\,,\;\; \mu_{2}(j)=|j_{p}|
 \;\;{\rm and}\;\; \s_{i}\s_{p}=+1\,.
\end{aligned}
\]
We define $G_{3}^{(-1)}:=G_3-G_{3}^{(+1)}$. 
\end{definition}
Consider the quadratic Hamiltonian 
$\widetilde{\mathcal{K}}^{(2)}_{\mathtt{m}}$ in \eqref{exampleHam}.
Given a trilinear Hamiltonian $G_{3}$ in $\mathcal{L}_{3}$
we define the adjoint action 
\[
{\rm ad}_{\widetilde{\mathcal{K}}^{(2)}_{\mathtt{m}}}G_3:=\{\widetilde{\mathcal{K}}^{(2)}_{\mathtt{m}}, G_3\}
\] 
 (see \eqref{Poisson}) 
as the Hamiltonian 
in $\mathcal{L}_{3}$ with coefficients 
\begin{equation}\label{adZ2}
\bullet {\bf (adjoint \; action)}\qquad
({\rm ad}_{\widetilde{\mathcal{K}}^{(2)}_{\mathtt{m}}}G_3)_{\s,j}:=
\Big(\ii \sum_{i=1}^{3}\s_i\omega({j_i})\Big)
(G_{3})_{\sigma,j}\,.
\end{equation}

The following lemma is the counterpart of Lemma $3.5$ in \cite{BFGI}.
We omit its proof.
\begin{lemma}\label{lem:hamvec}
Let $N\geq1$, 
$q_i\in\R$, 
consider 
$G^i_{3}(u)$ in $\mathcal{L}_{3}$. Assume
that the coefficients 
$(G^i_{3})_{\sigma,j}$ satisfy (recall Def. \ref{def:mumumu})
\begin{equation*}
|(G^i_{3})_{\sigma,j}|\leq C_{i} 
\mu_{3}(j)^{\beta_i}\mu_{1}(j)^{-q_{i}}\,,
\qquad \forall \s\in\{-1,+1\}^{3}\,,\; j\in \mathbb{Z}^{d}\setminus\{0\}\,,
\end{equation*}
for some $\beta_i>0$ and $C_i>0$, $i=1,2$.

\noindent {\bf (i) (Estimates on  Sobolev spaces).} 
Set $\delta=\delta_i$, $q=q_i$, $\beta=\beta_i$, $C=C_i$ 
and $G^i_{3}=G_3$ for $i=1,2$. There is $s_0=s_0(\beta,d)$ such that 
 for $s\geq s_0$,  $G_{3}$ defines naturally a smooth function
 from $H_{0}^{s}(\mathbb{T}^{d};\mathbb{C})$ to $\mathbb{R}$. In particular
 one has the following estimates: 
\begin{align*}
|G_{3}(w)|
&\lesssim_{s}C\|w\|_{H^{s}}^{3}\,,
\\
\|X_{G_3}(w)\|_{H^{s+q}}
&\lesssim_s C\|w\|_{H^{s}}^{2}\,,
\\
\|X_{G_{3}^{>N}}(w)\|_{H^{s}}
&\lesssim_{s} CN^{-q}\|w\|^{2}_{H^{s}}\,, 
\end{align*}
for any $w\in H_{0}^{s}(\mathbb{T}^{d};\mathbb{C}).$ 

\noindent{\bf (ii) (Poisson bracket).} The Poisson bracket between 
$G^1_{3}$ and $G^{2}_{3}$ 
satisfies the estimate 
\begin{equation*}
|\{G_{3}^{1},G_{3}^{2}\}|\lesssim_{s} C_1 C_2
\|w\|_{H^{s}}^{4}\,.
\end{equation*}
Let $F: H_0^{s}(\mathbb{T}^{d};\mathbb{C})\to \mathbb{R}$ a $C^{1}$ Hamiltonian function
such that
\[
\|X_{F}(w)\|_{H^{s}}\lesssim_{s}C_{3}\|w\|_{H^{s}}^{3}\,,
\] 
for some $C_{3}>0$.
Then one has
\[
|\{G_{3}^{1},F\}|\lesssim_{s} C_1 C_3\|w\|_{H^{s}}^{5}\,.
\]
\end{lemma}
We have the following result.

\begin{lemma}{\bf (Energy estimate).}\label{lem:energia}
Consider the Hamiltonians $N_{s}$ in \eqref{HamNN}, 
$G_3\in \mathcal{L}_{3}$
and write $G_{3}=G_{3}^{(+1)}+G_{3}^{(-1)}$ (recall Definition \ref{Ham:class}). 
Assume also that  the coefficients
of $G_3$ satisfy
\begin{equation*}
|(G^{(\eta)}_{3})_{\sigma,j}|\leq C
\mu_{3}(j)^{\beta}\mu_{1}(j)^{-q}\,,
 \qquad \forall \s\in\{-1,+1\}^{3}\,,\; j\in \mathbb{Z}^{d}\setminus\{0\}\,,\eta\in\{-1,+1\},
\end{equation*}
for some $\beta>0$, $C>0$ and $q\geq0$.  
We have that the Hamiltonian $Q_{3}^{(\eta)}:=\{N_s,G_{3}^{(\eta)}\}$, 
$\eta\in\{-1,1\}$,
belongs to the class $\mathcal{L}_{3}$ 
and has coefficients satisfying 
\begin{equation*}
|(Q_{3}^{(\eta)})_{\sigma,j}|\lesssim_{s} 
C \mu_{3}(j)^{\beta+1}\mu_1(j)^{2s}
\mu_1(j)^{-q-\alpha}\,,\qquad 
\alpha:=\left\{
\begin{aligned}
& 1\,,\;\;\; {\rm if }\;\; \eta=-1\\
&0\,, \;\;\; {\rm if}\;\; \eta=+1\,.
\end{aligned}
\right.
\end{equation*}
\end{lemma}

\begin{proof}
One can  reason as in the proof of  Lemma $4.2$ in \cite{BFGI}.
\end{proof}

\begin{remark}
As a consequence of 
Lemma \ref{lem:energia} we have the following. 
The action of the operator $\{N_{s},\cdot\}$
on multilinear  Hamiltonian functions as in \eqref{HamG}
where the two highest indexes have opposite sign (i.e. $G_{3}^{(-1)}$),
provides  a decay property of the coefficients w.r.t. the highest index.
This implies  (by Lemma \ref{lem:hamvec}-$(ii)$) a smoothing property of the Hamiltonian 
$\{N_{s},G_{3}^{(-1)}\}$.
\end{remark}

\subsection{Proof of Proposition \ref{modifiedstep}}\label{sec:modstep}
Recalling Definitions \ref{Ham:class}, \ref{Ham:class2} 
and considering the Hamiltonian $\widetilde{\mathcal{K}}^{(3)}_{\mathtt{m}}$
in \eqref{exampleHam}, \eqref{HamKK},
we write $\widetilde{\mathcal{K}}^{(3)}_{\mathtt{m}}=\widetilde{\mathcal{K}}^{(3,+1)}_{\mathtt{m}}+
\widetilde{\mathcal{K}}^{(3,-1)}_{\mathtt{m}}$.
We define (see \eqref{adZ2})
\begin{equation}\label{def:EEk}
E_{3}^{(+1)}:=({\rm ad}_{\widetilde{\mathcal{K}}^{(2)}_{\mathtt{m}}})^{-1}
\{N_{s}, \widetilde{\mathcal{K}}^{(3,+1)}_{\mathtt{m}}\}\,,
\qquad E_{3}^{(-1)}:=({\rm ad}_{\widetilde{\mathcal{K}}^{(2)}_{\mathtt{m}}})^{-1} 
\{N_{s}, (\widetilde{\mathcal{K}}^{(3,-1)}_{\mathtt{m}})^{(\leq N)}\}\,,
\end{equation}
and we set $E_{3}:=E_{3}^{(+1)}+E_{3}^{(-1)}$.
It is easy to note that   $E_{3}\in \mathcal{L}_{3}$.
Moreover, using that 
$|(\widetilde{\mathcal{K}}^{(3)}_{\mathtt{m}})_{\s,j}|\lesssim1$ (see \eqref{howimet}), Lemma \ref{lem:energia}
and 
Proposition \ref{measures},
one can check that the coefficients 
$(E_{3})_{\s,j}$ satisfy the \eqref{coeffE6}.
Using  \eqref{def:EEk} we notice that
\begin{equation}\label{omoomoeq}
\{N_{s}, \widetilde{\mathcal{K}}^{(3)}_{\mathtt{m}}\}+\{E_{3},\widetilde{\mathcal{K}}^{(2)}_{\mathtt{m}}\}
=\{N_{s}, (\widetilde{\mathcal{K}}^{(3,-1)}_{\mathtt{m}})^{(>N)}\}\,.
\end{equation}
Combining Lemmata \ref{lem:hamvec} and 
\ref{lem:energia} we deduce 
\begin{equation}\label{stimasmooth}
|{\{N_{s}, (\widetilde{\mathcal{K}}^{(3,-1)}_{\mathtt{m}})^{(>N)}\}}(w)|\lesssim_{s}
N^{-1}\|w\|^{3}_{H^{s}}\,,
\end{equation}
for $s$ large enough with respect to $M$.
We now prove the estimate \eqref{energyestimate}.
We have
\begin{align}
\{N_{s}+E_{3},\widetilde{\mathcal{K}}_{\mathtt{m}}\}&\stackrel{\eqref{exampleHam}}{=}
\{N_{s}+E_3, \widetilde{\mathcal{K}}_{\mathtt{m}}^{(2)}+\widetilde{\mathcal{K}}_{\mathtt{m}}^{(3)}
+\widetilde{\mathcal{K}}_{\mathtt{m}}^{(\geq4)}\}
\\
&=\{N_{s},\widetilde{\mathcal{K}}_{\mathtt{m}}^{(2)}\}\label{energia1}\\
&+ \{N_{s}, \widetilde{\mathcal{K}}_{\mathtt{m}}^{(3)}\}+\{E_{3},\widetilde{\mathcal{K}}_{\mathtt{m}}^{(2)}\}
\label{energia3}\\
&
+\{E_3,\widetilde{\mathcal{K}}_{\mathtt{m}}^{(3)}+\widetilde{\mathcal{K}}_{\mathtt{m}}^{(\geq4)}\}
+\{N_s, \widetilde{\mathcal{K}}_{\mathtt{m}}^{(\geq4)}\}.
\label{energia5}
\end{align}
We study each summand separately.
First of all note that
(recall \eqref{HamNN}, \eqref{exampleHam}) 
the term \eqref{energia1} vanishes.
By 
\eqref{R8888}, \eqref{coeffE6} and Lemma \ref{lem:hamvec}-$(ii)$
we obtain
\[
|\eqref{energia5}|\lesssim_{s}N^{d-2}\log^{d+1}(1+N)\|w\|_{H^{s}}^{4}\,.
\]
Moreover, by \eqref{omoomoeq}, \eqref{stimasmooth}, we deduce
\[
|\eqref{energia3}|\lesssim_{s}
N^{-1}\|w\|^{3}_{H^{s}}\,.
\]
The discussion above implies the bound \eqref{energyestimate}.

\subsection{Proof of the main result}\label{sec:proofthm}
Consider the Hamiltonian $\widetilde{\mathcal{K}}_{\mathtt{m}}(w,\bar{w})$ in \eqref{exampleHam}
and the associated Cauchy problem
\begin{equation}\label{winston10}
\left\{\begin{aligned}
\ii\pa_{t}w&=\pa_{\bar{w}}\widetilde{\mathcal{K}}_{\mathtt{m}}(w,\bar{w})\\
w(0)&=w_0\in H_0^{s}(\mathbb{T}^{d};\mathbb{C})\,,
\end{aligned}\right.
\end{equation}
for some $s>0$ large.
We shall prove the following.
\begin{lemma}{\bf (Main bootstrap)}\label{main:boot}
Let $s_0=s_0(d)$ given by Proposition \ref{modifiedstep}.
For any  $s\geq s_0$, there exists 
$\eps_0=\eps_0(s)$ such that the following holds.
Let $w(t,x)$ be a solution of \eqref{winston10} with 
$t\in [0,T)$, $T>0$ and initial condition
$w(0,x)=w_0(x)\in H_0^{s}(\mathbb{T}^{d};\mathbb{C})$. 
For any $\eps\in (0, \eps_0)$
if
\begin{equation}\label{hypBoot}
\|w_0\|_{H^{s}}\leq \eps\,,\quad \sup_{t\in[0,T)}\|w(t)\|_{H^{s}}\leq 2\eps\,,
\quad T\leq \eps^{-1-\frac{1}{d-1}}\log^{-\frac{d+1}{2}}\big(1+\e^{\frac{1}{1-d}}\big)\,,
\end{equation}
then we have the improved bound 
$\sup_{t\in[0,T)}\|w(t)\|_{H^{s}}\leq 3/2 \eps$\,.
\end{lemma}

First of all we show that the energy $N_{s}+E_3$ constructed by Proposition \ref{modifiedstep}
provides an equivalent Sobolev norm.

\begin{lemma}{\bf (Equivalence of the energy norm).}\label{equivNorms}
Let  $N\geq 1$.
Let $w(t,x)$ as in \eqref{hypBoot} with $s\gg1$ large enough. 
 Then, for any $0<c_0<1$,  there exists $C=C(s,d,c_0)>0$ such that, if
 we have the smallness condition
\begin{equation}\label{smalleps}
\eps C N^{d-2}\log^{(d+1)}(1+N )\leq 1\,,
\end{equation}
the following holds true. Define
\begin{equation}\label{def:calE}
 \mathcal{E}_{s}(w):=(N_{s}+E_3)(w)
\end{equation}
with  $N_{s}$ is in \eqref{HamNN}, $E_3$ is given by Proposition \ref{modifiedstep}.
We have
\begin{equation}\label{equivEEE}
1/(1+c_0)\mathcal{E}_{s}(w)\leq \|w\|^{2}_{H^{s}}
\leq (1+c_0)\mathcal{E}_{s}(w)\,,\quad\forall t\in[0,T]\,.
\end{equation}
\end{lemma}

\begin{proof}
Fix $c_0>0$.
By  \eqref{coeffE6}
and Lemma \ref{lem:hamvec}, we deduce
\begin{equation}\label{stimaE(z)}
|E_3(w)|\leq \tilde{C} \|w\|_{H^{s}}^{3}N^{d-2}\log^{(d+1)}(1+N )\,,
\end{equation}
for some $\tilde{C}>0$ depending on $s$. 
Then, recalling \eqref{def:calE}, we get
\[
|\mathcal{E}_{s}(w)|\leq \|w\|^{2}_{H^{s}}
(1+ \tilde{C}\|w\|_{H^{s}}N^{d-2}\log^{(d+1)}(1+N ))
\stackrel{(\ref{smalleps})}{\leq }\|w\|_{H^{s}}^{2}(1+c_0)\,,
\]
where we have chosen $C$ in \eqref{smalleps} large enough.
This implies the first inequality in \eqref{equivEEE}.
On the other hand, using \eqref{stimaE(z)} 
and \eqref{hypBoot},  we have
\[
\begin{aligned}
\|w\|_{H^{s}}^{2}\leq \mathcal{E}_{s}(w)
+\tilde{C}
N^{d-2}\log^{(d+1)}(1+N )\eps\|w\|_{H^{s}}^{2}\,.
\end{aligned}
\]
Then,
taking $C$ in  \eqref{smalleps} large enough,
we obtain the second inequality in \eqref{equivEEE}.
\end{proof}

\begin{proof}[{\bf Proof of Lemma \ref{main:boot}}]
We study how the equivalent  
energy norm $\mathcal{E}_{s}(w)$ defined in \eqref{def:calE} evolves along the flow of
\eqref{winston10}.
Notice that
\[
\pa_{t}\mathcal{E}_{s}(w)=-\{\mathcal{E}_{s},\mathcal{H}\}(w)\,.
\]
Therefore, for any $t\in [0,T]$, we have that
\[
\begin{aligned}
\left|\int_{0}^{T}\pa_{t}\mathcal{E}_{s}(w)\ \mathrm{d}t\right|
&\stackrel{\eqref{energyestimate},\eqref{hypBoot}}{\lesssim_{s}}
T N^{d-2}\log^{(d+1)}(1+N )\e^{4}
+N^{-1}\e^{3}\,.
\end{aligned}
\]
Let $0<\alpha$ and set $N:=\e^{-\alpha}$.
 Hence we 
have
\begin{align}
\left|\int_{0}^{T}\pa_{t}\mathcal{E}_{s}(w)\ \mathrm{d}t\right|&\lesssim_{s}
\eps^{2}T\big(\e^{2-\alpha(d-2)}\log^{(d+1)}(1+\e^{-\alpha} )
+\e^{1+\alpha}
\big)\label{moon1}\,.
\end{align}
We choose $\alpha>0$ such that
\begin{equation}\label{alpha1}
2-\alpha(d-2)=1+\alpha\,,\quad{\rm i.e.}\quad
\alpha:=\frac{1}{d-1}\,.
\end{equation}
Therefore estimate \eqref{moon1}
becomes
\[
\left|\int_{0}^{T}\pa_{t}\mathcal{E}_{s}(w)\ \mathrm{d}t\right|\lesssim_{s}
\eps^{2}T\eps^{\frac{d}{d-1}}\log^{d+1}(1+\e^{-\alpha} )\,.
\]
Since $\eps$ can be chosen small with respect to $s$, 
with this choices we get
\[
\left|\int_{0}^{T}\pa_{t}\mathcal{E}_{s}(w)\ \mathrm{d}t\right|
\leq\eps^{2}/4
\]
as long as 
\begin{equation}\label{logaloga}
T\leq \eps^{-d/(d-1)}\log^{-\frac{d+1}{2}}\big(1+\e^{\frac{1}{1-d}}\big)\,.
\end{equation}
Then, using the equivalence of norms  \eqref{equivEEE} and choosing 
$c_0>0$ small enough, we have
\[
\begin{aligned}
\|w(t)\|_{H^{s}}^{2}&\leq (1+c_0)\mathcal{E}_0(w(t))
\leq (1+c_0)\Big[\mathcal{E}_s(w(0))
+\left|\int_{0}^{T}\pa_{t}\mathcal{E}_{s}(w)\ \mathrm{d}t\right|\Big]
\\&
\leq (1+c_0)^{2}\eps^{2}+(1+c_0)\eps^{2}/4\leq \eps^{2}3/2\,,
\end{aligned}
\]
for times $T$ as in \eqref{logaloga}.
This implies the thesis.  
\end{proof}

\begin{proof}[\bf Proof of Theorem \ref{thm-main}]
 In the same spirit of  \cite{Feola-Iandoli-local-tori}, \cite{BMM1}
we have that for any initial condition $(\rho_0,\phi_0)$ as in \eqref{initialstima}
there exists a solution of \eqref{EK3} satisfying 
  \[
  \sup_{t\in[0,T)}\Big(\frac{1}{\mathtt{m}}\|\rho(t,\cdot)\|_{H^{s}} +
\frac{1}{\sqrt{k}}\|\Pi_0^{\bot} \phi(t,\cdot)\|_{H^{s}}    \Big)\leq 2\eps
  \]
  for some $T>0$ possibly small.
The result follows by Lemma \ref{main:boot}. By Lemma \ref{federico2} and estimates \eqref{stimeC}
we deduce that the function $w$
solving the equation \eqref{linear22} is defined over the time interval $[0,T]$
and satisfies
\[
\sup_{t\in[0,T]}\|w(t)\|_{H^{s}}\leq4\sqrt{\mathtt{m}}(1+\sqrt{k}\beta)\e\,.
\]
As long as $\nu\in [1,2]^{d}$ (defined as at the beginning of section \ref{sec:meas})
 belongs to the full Lebesgue measure set given by Proposition \ref{measures}, we can apply 
 Proposition \ref{modifiedstep} if $\e$ is small enough. Then
 by Lemma \ref{main:boot} and by a standard bootstrap argument
 we deduce that the solution $w(t)$ is defined for $t\in [0,T_{\e}]$, 
 $T_{\e}$ as in \eqref{tesiKG},
 and
 \[
\sup_{t\in[0,T_{\e}]}\|w(t)\|_{H^{s}}\leq8\sqrt{\mathtt{m}}(1+\sqrt{k}\beta)\e\,.
\]
 Using again Lemma \ref{federico2} and \eqref{stimeC} one can deduce the bound \eqref{tesiKG}.
 Hence the thesis follows.
\end{proof}

\def\cprime{$'$}

\end{document}